\documentclass[1 leqno,11pt,psfig]{amsart}
\usepackage{amssymb, amsmath,amsmath,latexsym,amssymb,amsfonts,amsbsy, amsthm,mathtools,graphicx,CJKutf8,CJKnumb,CJKulem,color}
\usepackage{graphicx,epsfig}
\usepackage{graphicx}
\usepackage{amssymb}
\usepackage{graphicx}
\usepackage{graphicx,epsfig}
\usepackage{amsmath}
\usepackage{mathrsfs}
\usepackage{amssymb}
\usepackage{dsfont}

\usepackage{amssymb,mathrsfs,amsfonts,amsmath,amsbsy,tikz}\usepackage{fontenc}\usepackage{textcomp}

\numberwithin{equation}{section}

\allowdisplaybreaks

\newcommand{\beq}{\begin{equation}}
\newcommand{\eeq}{\end{equation}}
\newcommand{\ben}{\begin{eqnarray}}
\newcommand{\een}{\end{eqnarray}}
\newcommand{\beno}{\begin{eqnarray*}}
\newcommand{\eeno}{\end{eqnarray*}}

\newtheorem{theorem}{Theorem}[section]
\newtheorem{definition}[theorem]{Definition}
\newtheorem{lemma}[theorem]{Lemma}
\newtheorem{proposition}[theorem]{Proposition}
\newtheorem{corollary}[theorem]{Corollary}

\newtheorem{remark}[theorem]{Remark}

\begin{document}
\title[Topological sequence entropy]{Topological sequence entropy of nonautonomous dynamical systems}

\author{Hua Shao}
\address{Department of Mathematics, Nanjing University of Aeronautics and Astronautics,
  Nanjing 211106, P. R. China}
\address{Key Laboratory of Mathematical Modelling and High Performance Computing of Air Vehicles, Nanjing
University of Aeronautics and Astronautics, MIIT, Nanjing 211106, P. R. China}
\email{huashao@nuaa.edu.cn}
\date{\today}

\maketitle

\begin{abstract}
Let $f_{0,\infty}=\{f_n\}_{n=0}^{\infty}$ be a sequence of continuous self-maps on a compact metric space $X$.
Firstly, we obtain the relations between topological sequence entropy of a nonautonomous dynamical system
$(X,f_{0,\infty})$ and that of its finite-to-one extension. We then prove that the topological sequence entropy of
$(X,f_{0,\infty})$ is no less than its corresponding measure sequence entropy if $X$ has finite
covering dimension. Secondly, we study the supremum topological sequence entropy of $(X,f_{0,\infty})$, and confirm that
it equals to that of its $n$-th compositions system if $f_{0,\infty}$ is equi-continuous; and we prove the
supremum topological sequence entropy of $(X,f_{i,\infty})$ is no larger than that of $(X,f_{j,\infty})$ if $i\leq j$,
and they are equal if $f_{0,\infty}$ is equi-continuous and surjective. Thirdly, we investigate the
topological sequence entropy relations between $(X,f_{0,\infty})$ and $(\mathcal{M}(X),\hat{f}_{0,\infty})$ induced on
the space $\mathcal{M}(X)$ of all Borel probability measures, and obtain that given any sequence, the topological
sequence entropy of $(X,f_{0,\infty})$ is zero if and only if that of $(\mathcal{M}(X),\hat{f}_{0,\infty})$ is zero;
the topological sequence entropy of $(X,f_{0,\infty})$ is positive if and only if that of $(\mathcal{M}(X),\hat{f}_{0,\infty})$
is infinite. By applying this result, we obtain some big differences between entropies of
nonautonomous dynamical systems and that of autonomous dynamical systems. Finally, we study whether
multi-sensitivity of $(X,f_{0,\infty})$ imply positive or infinite topological sequence entropy.
\end{abstract}



\section{Introduction}
Entropy is a key tool to measure the complexity of dynamical systems. The measure-theoretical entropy
was first introduced by Kolmogorov, then Adler, Konhelm and McAndrew extended it to the
context of topological dynamical systems. In order to establish the relations between spectrum theory
and entropy theory, Ku\v shnirenko \cite{Kushnirenko} generalized measure-theoretical entropy and
proposed the concept of sequence entropy. In a precisely analogous way, Goodman \cite{Goodman74}
introduced the notion of topological sequence entropy as an extension of topological entropy.
Topological sequence entropy is an important conjugacy invariant, which is useful
for distinguishing between chaotic and nonchaotic interval maps \cite{Franz91},
and also a good tool to distinguish among continuous maps with zero topological entropy.

In 1996, Kolyada and Snoha \cite{Kolyada96} extended the concept of topological entropy to
nonautonomous dynamical systems, and obtained a series of important properties of it. Their major
motivation is to have a good understanding of the topological entropy of skew product of dynamical systems, whose
topological entropy can be ultimately transformed into computing the topological entropy of
a nonautonomous dynamical system. Note that the majority of complex systems in biology, physics and engineering
are driven by sequences of different functions, and thus the study on nonautonomous dynamical systems is of importance
in applications. Moreover, the behaviors of nonautonomous dynamical systems are much richer and sometimes quite
different than what are expected from the classical cases. For example, Balibrea and Oprocha \cite{Bali12}
constructed a nonautonomous dynamical system which has positive topological entropy, but does not exhibit Li-Yorke chaos.
Based on the above reasons, we shall study topological sequence entropy (t.s.e. for short)
of nonautonomous dynamical systems (NDSs for short) in the present paper.

By a NDS, we mean a pair $(X,f_{0,\infty})$, where $X$ is a compact metric space equipped with a metric $d$ and
$f_{0,\infty}=\{f_n\}_{n=0}^{\infty}$ is a sequence of continuous self-maps on $X$. If $f_n=f$ for each $n\geq0$,
then $(X, f_{0,\infty})$ becomes the classical dynamical system $(X,f)$. For any $x_0\in X$, the positive orbit
$\{x_n\}_{n=0}^{\infty}$ of $(X,f_{0,\infty})$ starting from $x_0$ is defined by $x_n=f_0^n(x_0)$, where
$f_0^n=f_{n-1}\circ\cdots\circ f_{0}$ for any $n\geq 1$, and $\{x_n\}_{n=0}^{\infty}$ can be seen as a solution
of the nonautonomous difference equation
\[x_{n+1}=f_n(x_n),\;n\geq0.\]
Now, many authors have studied entropy theory of NDSs.
For example, Zhu et al. \cite{Zhu12} investigated the properties of entropies (both topological and measure-theoretic) in NDSs,
and also discussed their relations. Kawan \cite{Kawan13} studied measure-theoretical entropy
in NDSs, and proved that the measure-theoretical entropy can be estimated from above by its topological entropy.
Andres \cite{Andres22} studied parametric topological entropy for a sequence of continuous self-maps on a compact metric space,
and applied it to differential equations with time-dependent impulses. Recently, we \cite{Shao23} considered the properties,
calculations and estimations of topological entropy for a NDS on a compact uniform space.
For more information about entropy theory of NDSs, the readers are referred to \cite{Andres22,Bis,Canovas13,Kawan13,Kolyada96,Kolyada99,Liu20,Shao16,Shao22,Shao23,Sotola,Xu18,Zhu12}
and references therein.

It is worth mentioning that Bauer and Sigmund \cite{Bauer} studied systematically what topological properties of
$(X,f)$ can carry over to its induced system $(\mathcal{M}(X),\hat{f})$ on the space $\mathcal{M}(X)$ of all
Borel probability measures. They proved that positive topological entropy of $(X,f)$ implies infinite topological entropy
of $(\mathcal{M}(X),\hat{f})$. Glasner and Weiss \cite{Glasner} obtained that $(X,f)$ has zero topological entropy
if and only if $(\mathcal{M}(X),\hat{f})$ has zero topological entropy. The method in one of their proof established
a remarkable connection between the entropy theory and combinatorial theory.
Then, Kerr and Li \cite{Kerr05} further developed this method and obtained that $(X,f)$ is null if and only if
$(\mathcal{M}(X),\hat{f})$ is null, and recall that $(X,f)$ is null if the t.s.e. of $(X,f)$ is zero along any sequence.
Later, Qiao and Zhou \cite{Qiao17} strengthened this result and extended Glasner and Weiss's result
to t.s.e. Recently, Liu et al. \cite{Liu20} extended Glasner and Weiss's result
to NDSs. Motivated by these work, we consider the t.s.e. relations between a NDS and its induced system on
$\mathcal{M}(X)$, which is one problem studied in the present paper.
Some main results are listed as follows.

(1) The t.s.e. relation between a NDS and its finite-to-one extension is given, see Theorem \ref{8244}.

(2) The t.s.e. of $(X,f_{0,\infty})$ is larger than or equal to its corresponding measure sequence entropy
if $X$ has finite covering dimension, see Theorem \ref{7281}.

(3) The supremum t.s.e. of $(X,f_{0,\infty})$ equals to that of its $n$-th compositions system provided that
$f_{0,\infty}$ is equi-continuous, see Theorem \ref{sequence entropy equality}.

(4) The supremum t.s.e. of $(X,f_{i,\infty})$ is less than or equal to that of $(X,f_{j,\infty})$ if $i\leq j$,
and they are equal if $f_{0,\infty}$ is equi-continuous and surjective,
see Propositions \ref{sequence entropy inequality} and \ref{8246}.

(5) For any sequence $A$, the t.s.e. $h_A(X, f_{0,\infty})=0$ if and only if $h_A(\mathcal{M}(X), \hat{f}_{0,\infty})=0$;
and $h_A(X, f_{0,\infty})>0$ if and only if $h_A(\mathcal{M}(X),\hat{f}_{0,\infty})=+\infty$; see Theorems \ref{sequence entropy}
and \ref{sequence entropy 1}.

(6) If $(X,f)$ has positive topological entropy, then $h_A(X, f)>0$ and $h_A(\mathcal{M}(X),\hat{f})=+\infty$
for any sequence $A$. However, there exists a NDS $(X, f_{0,\infty})$ with positive topological entropy, but
$h_A(X, f_{0,\infty})=h_A(\mathcal{M}(X),\hat{f}_{0,\infty})=0$ along some sequence $A$, see Theorem \ref{822}.

(7) For any topologically transitive interval dynamical system $(I,f)$, $h_A(\mathcal{M}(X),\hat{f})=+\infty$
for any sequence $A$. However, there exists a topologically transitive NDS $(I,f_{0,\infty})$
satisfying that $h_A(\mathcal{M}(X),\hat{f}_{0,\infty})=0$ along some sequence $A$, see Theorem \ref{8242}.

(8) If $(X,f_{0,\infty})$ is multi-sensitive, then there exists some sequence $A$ such that  $h_A(X,f_{0,\infty})=h_A(\mathcal{M}(X),\hat{f}_{0,\infty})=+\infty$ provided that $X$ is locally connected
and $f_0$ is surjective, see Theorem \ref{8182}.

Now, we compare some existing results in the literatures with our new results in this paper.
It is shown in \cite{Bali99} that Bowen's inequality for t.s.e. does not hold.
In Theorem \ref{8244}, we obtain an inequality that characterizes the t.s.e. relation between a NDS and
its finite-to-one extension replacing Bowen's inequality.
Goodman \cite{Goodman74} showed that t.s.e. of $(X,f)$ on a finite-dimensional space is no less than
its corresponding measure sequence entropy with respect to any invariant measure. We extend this result
to NDSs and obtained that the t.s.e. of $(X,f_{0,\infty})$ bounds its measure sequence entropy in Theorem 3.7.
We generalize the notion of maximal pattern entropy,
which equals to the supremum of all t.s.e., introduced in \cite{Huang09} to NDSs. We get that the supremum
t.s.e. of $(X,f_{0,\infty})$ equals to that of its $n$-th compositions system provided that $f_{0,\infty}$
is equi-continuous in Theorem 4.4, which also extends the corresponding results in \cite{Huang09} and \cite{Wu17}
to NDSs. Moreover, we investigate whether the supremum t.s.e. of $(X,f_{i,\infty})$ increases or decreases
as $i$ increases in Propositions 4.5 and 4.6.
In Theorems \ref{sequence entropy} and \ref{sequence entropy 1}, we confirm that the t.s.e. relations between
$(X,f)$ and $(\mathcal{M}(X),\hat{f})$ in \cite{Qiao17} and  topological entropy relations between
$(X,f_{0,\infty})$ and $(\mathcal{M}(X),\hat{f}_{0,\infty})$ in \cite{Liu20} also hold for t.s.e. in the context of NDSs.
In particular, by applying Theorems \ref{sequence entropy} and \ref{sequence entropy 1}, we obtain
some big differences between sequence entropy of NDSs and that of classical (i.e. autonomous) dynamical systems,
see Theorems \ref{822} and \ref{8242}, and Proposition \ref{8243}. Huang et al. \cite{Huang16} showed that
any multi-sensitive system $(X,f)$ has positive topological sequence entropy. Then, Wu \cite{Wu17}
proved that the multi-sensitivity of $(X,f)$ implies infinite supremum t.s.e.if $X$ is locally connected.
Recently, Huang and Zhu \cite{Huang23} studied this problem in the context of group actions.
In Theorems 6.1 and 6.2, we explore this problem in NDSs and
obtain some NDS versions of the above mentioned results.

The rest of the paper is organized as follows. In Section 2, the definitions and basic properties of t.s.e. are presented, and the induced system on $\mathcal{M}(X)$ are also recalled. Several calculations and estimations of t.s.e. of $(X,f_{0,\infty})$ are investigated in Section 3. In Section 4, some properties of the supremum t.s.e. of $(X,f_{0,\infty})$ are studied. The relations between zero t.s.e. (resp. positive t.s.e.) of $(X,f_{0,\infty})$ and that of $(\mathcal{M}(X),\hat{f}_{0,\infty})$ are investigated in Section 5. Particularly, some interesting comparisons between t.s.e. of NDSs and that of autonomous dynamical systems are obtained. In Section 6, whether
multi-sensitivity of $(X,f_{0,\infty})$ imply positive or infinite t.s.e. are investigated.

\section{Preliminaries}

In this section, we first recall the definitions of t.s.e. using open covers
and using separated sets and spanning sets for $(X,f_{0,\infty})$, respectively, then present
some basic properties of it, and finally recall the the induced system of $(X,f_{0,\infty})$
on the space of probability measures.

By $\mathds{N}$ denote the set of positive integers and $\mathds{Z}_{+}=:\{0\}\cup \mathds{N}$,
and by $\mathcal{S}$ denote the set of all increasing sequences of $\mathds{Z}_{+}$.
Let $B(x,\epsilon)$ and $\bar{B}(x,\epsilon)$ denote the open and closed balls of radius $\epsilon>0$
centered at $x\in X$, respectively. Note that the cardinality of a set $F\subset X$ is denoted by $|F|$.

\subsection{Definitions of t.s.e.}

For any open covers $\mathscr{A}_1,\cdots,\mathscr{A}_n,\mathscr{A}$ of $X$, denote
\[\bigvee_{i=1}^{n}\mathscr{A}_i=\left\{\bigcap_{i=1}^{n}A_i: A_i\in\mathscr{A}_i,\; 1\leq i\leq n\right\},\;
f_{i}^{-n}(\mathscr{A})=\{f_{i}^{-n}(A): A\in\mathscr{A}\},\]
where
\[f_{i}^{n}=f_{i+n-1}\circ\cdots\circ f_{i},\;f_{i}^{-n}=(f_{i}^{n})^{-1}=f_{i}^{-1}\circ\cdots\circ f_{i+n-1}^{-1}.\]
Note that $f_{i}^{-1}$ will be applied to sets, we don't assume that $f_i$ is invertible for each $i\geq0$.
Let $\mathcal{N}(\mathscr{A})$ be the minimal possible cardinality of all subcovers chosen from $\mathscr{A}$.
Let $\Lambda$ be a nonempty subset of $X$. Denote the cover $\{A\cap\Lambda: A\in\mathscr{A}\}$ of the set $\Lambda$ by
$\mathscr{A}|_{\Lambda}$. Then the t.s.e. of $f_{0,\infty}$ on the set $\Lambda$ with respect to $\mathscr{A}$ and $A=\{a_i\}_{i=1}^{\infty}\in\mathcal{S}$ is defined by
\[h_A(f_{0,\infty},\Lambda,\mathscr{A})
=\limsup_{n\to\infty}\frac{1}{n}\log\mathcal{N}\left(\left(\bigvee_{i=1}^{n}f_{0}^{-a_i}\mathscr{A}\right)\big|_{\Lambda}\right),\]
and the t.s.e. of $f_{0,\infty}$ on $\Lambda$ along $A$ is defined by
\[h_A(f_{0,\infty},\Lambda)=\sup_{\mathscr{A}}h_A(f_{0,\infty},\Lambda,\mathscr{A}),\]
where supremum is taken over all finite open covers of $X$.

Now, we recall the definitions of t.s.e. of $(X,f_{0,\infty})$ based on separated sets and spanning sets.
Let $n\geq1$ and $\epsilon>0$. A set $E\subset X$ is called $(n,\epsilon,A)$-separated if, for any $x\neq y\in E$, there exists
$1\leq j\leq n$ such that $d(f_{0}^{a_j}(x),f_{0}^{a_j}(y))>\epsilon$. A set $F\subset X$ is called $(n,\epsilon,A)$-spans
another set $K\subset X$ if for any $y\in K$, there exists $x\in F$ such that $d(f_{0}^{a_j}(x),f_{0}^{a_j}(y))\leq\epsilon$ for any $1\leq j\leq n$. Let $s_{n}(\epsilon,A,f_{0,\infty},\Lambda)$ be the maximal cardinality of an $(n,\epsilon,A)$-separated set
in $\Lambda$, and $r_n(\epsilon,A,f_{0,\infty},\Lambda)$ and $r_n^{X}(\epsilon,A,f_{0,\infty},\Lambda)$
be the minimal cardinality of a set in $\Lambda$ and in $X$, respectively, which $(n,\epsilon,A)$-spans
$\Lambda$.
Since $X$ is compact, $s_{n}(\epsilon,A,f_{0,\infty},\Lambda)$, $r_n(\epsilon,A,f_{0,\infty},\Lambda)$ and $r_n^{X}(\epsilon,A,f_{0,\infty},\Lambda)$ are finite. It is easy to verify the following inequality:
\begin{equation}\label{72821}
r_n(\epsilon,A,f_{0,\infty},\Lambda)\leq s_n(\epsilon,A,f_{0,\infty},\Lambda)\leq r_n^{X}(\epsilon/2,A,f_{0,\infty},\Lambda)
\leq r_n(\epsilon/2,A,f_{0,\infty},\Lambda).
\end{equation}
By using separated sets and spanning sets, together with (\ref{72821}), one can define the t.s.e. of $f_{0,\infty}$ on $\Lambda$ along the sequence $A$ as
\begin{align*}\label{topological sequence entropy}
ent_{A}(f_{0,\infty},\Lambda):&=\lim_{\epsilon\to0}\limsup_{n\to\infty}\frac{1}{n}\log s_{n}(\epsilon,A,f_{0,\infty},\Lambda)
=\lim_{\epsilon\to0}\limsup_{n\to\infty}\frac{1}{n}\log r_{n}(\epsilon,A,f_{0,\infty},\Lambda)\\
&=\lim_{\epsilon\to0}\limsup_{n\to\infty}\frac{1}{n}\log r_{n}^X(\epsilon,A,f_{0,\infty},\Lambda).
\end{align*}

The diameter of an open cover $\mathscr{A}$ of $X$ is $d(\mathscr{A})=\sup_{A\in\alpha}d(A)$, where $d(A)=\sup_{x,y\in A}d(x,y)$.
A Lebesgue number of $\mathscr{A}$ is some $\delta>0$ satisfying that for any $E\subset X$, $d(E)<\delta$ implies that $E\subset A$
for some $A\in\mathscr{A}$. Let $\epsilon>0$ and $\mathscr{A}$ be an open cover of $X$ with diameter less than $\epsilon$. Then
\begin{equation}\label{211}
s_n(\epsilon,A,f_{0,\infty},\Lambda)\leq\mathcal{N}\left(\left(\bigvee_{i=1}^{n}f_{0}^{-a_i}\mathscr{A}\right)\big|_{\Lambda}\right).
\end{equation}
Let $\mathscr{A}$ be an open cover of $X$ with Lebesgue number $\delta>0$, then one obtains that
\begin{equation}\label{221}
\mathcal{N}\left(\left(\bigvee_{i=1}^{n}f_{0}^{-a_i}\mathscr{A}\right)\big|_{\Lambda}\right)\leq r_n(\delta/3,A,f_{0,\infty},\Lambda).
\end{equation}
It follows from (\ref{211}) and (\ref{221}) that we directly get the below result.
\begin{proposition}\label{724}
\[h_A(f_{0,\infty},\Lambda)=ent_{A}(f_{0,\infty},\Lambda).\]
\end{proposition}

\begin{remark}
By Proposition \ref{724}, all the above definitions of t.s.e. for $(X,f_{0,\infty})$ are equivalent.
In the following, we no longer distinguish between $h_A(f_{0,\infty},\Lambda)$ and $ent_{A}(f_{0,\infty},\Lambda)$,
and uniformly use $h_A(f_{0,\infty},\Lambda)$ to denote the t.s.e. of $f_{0,\infty}$ on the set
$\Lambda$ along the sequence $A$. If $\Lambda=X$, then $h_A(f_{0,\infty}):=h_{A}(f_{0,\infty},X)$ is called the t.s.e.
of $(X,f_{0,\infty})$ along the sequence $A$. If $A=\mathds{Z}_{+}$, then we recover standard topological entropy $h(f_{0,\infty})$ of $(X,f_{0,\infty})$ \cite{Kolyada96}.
\end{remark}

\subsection{Some basic properties of t.s.e.}

\begin{proposition}
Let $A=\{a_i\}_{i=1}^{\infty}\in\mathcal{S}$. If $X=\bigcup_{i=1}^{m}K_i$, then
\[h_{A}(f_{0,\infty})=\max_{1\leq i\leq m}h_{A}(f_{0,\infty},K_i).\]
\end{proposition}

\begin{proof}
Since $h_{A}(f_{0,\infty})\geq h_{A}(f_{0,\infty},K_i)$ for each $1\leq i\leq m$,
$h_{A}(f_{0,\infty})\geq\max_{1\leq i\leq m}h_{A}(f_{0,\infty},K_i)$ is trivial.
It is to show the reverse inequality.
Given an open cover $\mathscr{A}$ of $X$ and $n\geq1$. Denote $\mathscr{B}_1,\cdots,\mathscr{B}_m$ be subcovers chosen from the
covers $\big(\bigvee_{i=1}^{n}f_{0}^{-a_i}\mathscr{A}\big)|_{K_1},\cdots,\big(\bigvee_{i=1}^{n}f_{0}^{-a_i}\mathscr{A}\big)|_{K_m}$.
Let $\mathscr{B}=\bigcup_{i=1}^{m}\mathscr{B}_i$. Then $\mathscr{B}$ is also an open cover of $X$ and also a refinement
of $\bigvee_{i=1}^{n}f_{0}^{-a_i}\mathscr{A}$. Thus,
\[\mathcal{N}\left(\bigvee_{i=1}^{n}f_{0}^{-a_i}\mathscr{A}\right)
\leq\sum_{i=1}^{m}\mathcal{N}\left(\big(\bigvee_{i=1}^{n}f_{0}^{-a_i}\mathscr{A}\big)\big|_{K_i}\right).\]
This, together with Lemma 4.1.9 in \cite{Alse93}, implies that
\begin{equation*}
\begin{split}
h_A(f_{0,\infty},\mathscr{A})
&=\limsup_{n\to\infty}\frac{1}{n}\log\mathcal{N}\left(\bigvee_{i=1}^{n}f_{0}^{-a_i}\mathscr{A}\right)
\leq\limsup_{n\to\infty}\frac{1}{n}\log\sum_{i=1}^{m}\mathcal{N}\left(\big(\bigvee_{i=1}^{n}f_{0}^{-a_i}\mathscr{A}\big)\big|_{K_i}\right)\\
&=\max_{1\leq i\leq m}\limsup_{n\to\infty}\frac{1}{n}\log\mathcal{N}\left(\big(\bigvee_{i=1}^{n}f_{0}^{-a_i}\mathscr{A}\big)\big|_{K_i}\right)
=\max_{1\leq i\leq m}h_{A}(f_{0,\infty},K_i).
\end{split}
\end{equation*}
\end{proof}

\begin{proposition}
Let $A=\{a_i\}_{i=1}^{\infty}\in\mathcal{S}$ and $\{\alpha_n\}_{n=0}^{\infty}$
be a sequence of open covers of $X$ with $d(\alpha_n)\to0$ as $n\to\infty$.
Then
\[\lim_{n\to\infty}h_{A}(f_{0,\infty},\alpha_n)=h_{A}(f_{0,\infty}),\]
and consequently
\[h_{A}(f_{0,\infty})=\lim_{\delta\to0}\{\sup h_{A}(f_{0,\infty},\alpha): d(\alpha)<\delta\}.\]
\end{proposition}

\begin{proof}
Suppose that $h_{A}(f_{0,\infty})<\infty$. For any $\epsilon>0$, there exists an open cover $\beta$ such that
$h_{A}(f_{0,\infty},\beta)>h_{A}(f_{0,\infty})-\epsilon$. Let $\delta>0$ be a Lebesgue number of $\beta$.
Since $d(\alpha_n)\to0$ as $n\to\infty$, there exists $N\in\mathbf{Z^{+}}$ such that $d(\alpha_n)<\delta$
for any $n\geq N$. Then $\alpha_n$ is a refinement of $\beta$ for any $n\geq N$. Thus,
\[h_{A}(f_{0,\infty})-\epsilon<h_{A}(f_{0,\infty},\beta)\leq h_{A}(f_{0,\infty},\alpha_n)\leq h_{A}(f_{0,\infty}),\;n\geq N.\]
Hence, $\lim_{n\to\infty}h_{A}(f_{0,\infty},\alpha_n)=h_{A}(f_{0,\infty})$.

Suppose that $h_{A}(f_{0,\infty})=\infty$. For any $M>0$, there exists an open cover $\gamma$ such that
$h_{A}(f_{0,\infty},\gamma)>M$. Let $\delta>0$ be a Lebesgue number of $\gamma$. By the fact that $d(\alpha_n)\to 0$ as $n\to\infty$,
there exists $N\in\mathbf{Z^{+}}$ such that $d(\alpha_n)<\delta$ for any $n\geq N$. Thus,
\[h_{A}(f_{0,\infty},\alpha_n)\geq h_{A}(f_{0,\infty},\gamma)>M,\;n\geq N,\]
which implies that $h_{A}(f_{0,\infty},\alpha_n)\to\infty$ as $n\to\infty$.
\end{proof}

Given another system $(Y,g_{0,\infty})$, where $Y$ is a compact metric space equipped with a metric $\rho$
and $g_{0,\infty}=\{g_{n}\}_{n=0}^{\infty}$ is a sequence of continuous self-maps on $Y$. The metric on $X\times Y$ is defined by
\[d\times\rho((x_1,y_1),(x_2,y_2))=\max\{d(x_1,x_2),\rho(y_1,y_2)\},\;(x_1,y_1),(x_2,y_2)\in X\times Y.\]

\begin{proposition}
Let $A\in\mathcal{S}$. Then
\[\max\{h_{A}(f_{0,\infty}),h_{A}(g_{0,\infty})\}\leq h_{A}(f_{0,\infty}\times g_{0,\infty})\leq h_{A}(f_{0,\infty})+h_{A}(g_{0,\infty}).\]
\end{proposition}

\begin{proof}
Let $A=\{a_i\}_{i=1}^{\infty}$, $\epsilon>0$ and $m\geq1$. Suppose that $F_1$ is a set in $X$ which $(m,\epsilon,A)$-spans $X$ under $f_{0,\infty}$ with $|F|=r_m(\epsilon,A,f_{0,\infty})$ and $F_2$ is a set in $Y$ which $(m,\epsilon,A)$-spans $Y$ under $g_{0,\infty}$ with $|F_2|=r_m(\epsilon,A,g_{0,\infty})$. Then
$F_1\times F_2$ is a set in $X\times Y$ which $(m,\epsilon,A)$-spans $X\times Y$ under $f_{0,\infty}\times g_{0,\infty}$.
In fact, for any $(x,y)\in X\times Y$, there exist $z_1\in F_1$ and $z_2\in F_2$ such that $d(f_{0}^{a_j}(x),f_{0}^{a_j}(z_1))\leq\epsilon$
and $\rho(g_{0}^{a_j}(y),g_{0}^{a_j}(z_2))\leq\epsilon$ for any $1\leq j\leq m$, which yields that
\[d\times\rho((f_{0}^{a_j}(x),g_{0}^{a_j}(y)),(f_{0}^{a_j}(z_1),g_{0}^{a_j}(z_2))\big)\leq\epsilon,\;1\leq j\leq m.\]
Then
\[r_m(\epsilon,A,f_{0,\infty}\times g_{0,\infty})\leq r_m(\epsilon,A,f_{0,\infty})r_m(\epsilon,A,g_{0,\infty}).\]
Hence,
\[h_A(f_{0,\infty}\times g_{0,\infty})\leq h_A(f_{0,\infty})+h_A(g_{0,\infty}).\]

On the other hand, suppose that $E_1\subset X$ is $(m,\epsilon,A)$-separated under $f_{0,\infty}$ with
$|E_1|=s_m(\epsilon,A,f_{0,\infty})$ and $E_2\subset Y$ is $(m,\epsilon,A)$-separated under $g_{0,\infty}$
with $|E_2|=s_m(\epsilon,A,g_{0,\infty})$. Then $E_1\times E_2\subset X\times Y$ is $(m,\epsilon,A)$-separated
under $f_{0,\infty}\times g_{0,\infty}$. This, together with the fact that
\[|E_1\times E_2|\geq\max\{|E_1|,|E_2|\},\]
implies that
\[s_m(\epsilon,A,f_{0,\infty}\times g_{0,\infty})\geq\max\{s_m(\epsilon,A,f_{0,\infty}),s_m(\epsilon,A,g_{0,\infty})\}.\]
Therefore,
\[h_{A}(f_{0,\infty}\times g_{0,\infty})\geq\max\{h_{A}(f_{0,\infty}),h_{A}(g_{0,\infty})\}.\]
\end{proof}

\begin{definition}
If for each $n\geq0$, there exists a surjective map $\pi_n:X\to Y$ such that $\pi_{n+1}\circ f_n=g_n\circ \pi_n$, and
$\{\pi_n\}_{n=0}^{\infty}$ is equi-continuous, then $(X,f_{0,\infty})$ is said to be topologically $\{\pi_n\}_{n=0}^{\infty}$-equi-semiconjugate to $(Y,g_{0,\infty})$. Furthermore, if $\{\pi_n^{-1}\}_{n=0}^{\infty}$
is also equi-continuous, they are said to be topologically $\{\pi_n\}_{n=0}^{\infty}$-equi-conjugate.
\end{definition}

The following lemma shows that t.s.e. is also an invariant under topological equi-conjugacy.

\begin{lemma}\cite{Shao22}\label{22}
If $(X,f_{0,\infty})$ is topologically equi-semiconjugate to $(Y,g_{0,\infty})$, then $h_A(g_{0,\infty})\leq h_A(f_{0,\infty})$
for any $A\in\mathcal{S}$. Consequently, if $(X,f_{0,\infty})$ is topologically equi-conjugate to $(Y,g_{0,\infty})$,
then $h_A(g_{0,\infty})= h_A(f_{0,\infty})$.
\end{lemma}

\subsection{The induced system on the space of probability measures}

Let $\mathcal{B}(X)$ be the $\sigma$-algebra of all Borel subsets of $X$ and $\mathcal{M}(X)$ be the space of all
Borel probability measures on $X$. The Prohorov metric $\mathcal{P}_{d}$ on $\mathcal{M}(X)$ is defined by
\begin{align*}
\mathcal{P}_{d}(\mu,\nu)=\inf\{\epsilon>0: \mu(A)\leq\nu(A^{\epsilon})+\epsilon\;\;{\rm and}\; \nu(A)\leq\mu(A^{\epsilon})+\epsilon,
\;A\in\mathcal{B}(X)\}
\end{align*}
for $\mu,\nu\in\mathcal{M}(X)$, where $A^{\epsilon}=\{x\in X: d(x,A)<\epsilon\}$. It was proved in \cite{Strassen} that
\begin{align*}
\mathcal{P}_{d}(\mu,\nu)=\inf\{\epsilon>0: \mu(A)\leq\nu(A^{\epsilon})+\epsilon,\;A\in\mathcal{B}(X)\}.
\end{align*}
The topology induced by the metric $\mathcal{P}_{d}$ coincides with the weak$^{*}$-topology for measures.
$(X, f_{0,\infty})$ induces $(\mathcal{M}(X),\hat{f}_{0,\infty})$, where $\hat{f}_{0,\infty}=\{\hat{f}_n\}_{n=0}^{\infty}$,
with $\hat{f}_n: \mathcal{M}(X)\to \mathcal{M}(X)$ defined by
\begin{align}\label{1.1}
\hat{f}_n(\mu)(A)=\mu\big(f_{n}^{-1}(A)\big),\;\mu\in \mathcal{M}(X),\;A\in\mathcal{B}(X).
\end{align}
Denote $\hat{f}_{0}^{n}=\hat{f}_{n-1}\circ\cdots\circ\hat{f}_{0}$. Then (\ref{1.1}) ensures that
\begin{align*}
\hat{f}_{0}^{n}(\mu)(A)=\mu\big(f_{0}^{-n}(A)\big),\;\mu\in\mathcal{M}(X),\;A\in\mathcal{B}(X),\;n\geq1.
\end{align*}
$(\mathcal{M}(X),\mathcal{P}_{d})$ is a compact metric space and $\hat{f}_n$ is continuous on $\mathcal{M}(X)$ for
each $n\geq0$ (see \cite{Bauer,Parthasarathy} for more details).

Let $x\in X$ and $\delta_{x}\in\mathcal{M}(X)$ be the Dirac point measure of $x$ defined by $\delta_{x}(A)=1$ if $x\in A$
and $\delta_{x}(A)=0$ otherwise, for any $A\in \mathcal{B}(X)$. It is easy to see that $\mathcal{M}(X)$ is convex and
that the Dirac point measures are just the extremal points of $\mathcal{M}(X)$. Note that
\[\hat{f}_{n}(\delta_{x})=\delta_{f_{n}(x)},\; x\in X,\; n\geq0.\]
Denote
\[\mathcal{M}_{n}(X)\triangleq\left\{\frac{1}{n}\sum_{i=1}^{n}\delta_{x_i}: x_1,\cdots,x_n\in X\right\},\;n\geq1.\]
Then $\mathcal{M}_{n}(X)$ is closed in $\mathcal{M}(X)$ for any $n\geq1$, and $\mathcal{M}_{\infty}(X)\triangleq\cup_{n=1}^{\infty}\mathcal{M}_{n}(X)$ is dense in $\mathcal{M}(X)$ \cite{Bauer}.

\section{Estimations of t.s.e.}

In this section, we shall study some calculations and estimations of t.s.e. in NDSs.
Particularly, we show that t.s.e. bounds measure sequence entropy.

\subsection{Some estimations of t.s.e.}
There is a counter-example in \cite{Bali99} showing that Bowen's inequality for t.s.e. does not hold.
However, we can obtain the following result.

\begin{theorem}\label{8244}
Let $(X,f_{0,\infty})$ be topologically $\{\pi_i\}_{i=0}^{\infty}$-semiconjugate to $(Y,g_{0,\infty})$.
Assume that $\{\pi_i\}_{i=0}^{\infty}\subset\{\phi_1,\cdots,\phi_k\}$ for some $k\geq1$ and for each fibre
$\phi_j^{-1}(y)$, $y\in Y$, $1\leq j\leq k$, has at most $n$ points. Then for any $A\in\mathcal{S}$,
\[h_{A}(g_{0,\infty})\leq h_{A}(f_{0,\infty})\leq h_{A}(g_{0,\infty})+\log nk.\]
\end{theorem}

\begin{proof}
Fix $A=\{a_i\}_{i=1}^{\infty}\in\mathcal{S}$.
By Lemma \ref{22}, $h_{A}(g_{0,\infty})\leq h_{A}(f_{0,\infty})$.
It is to show $h_{A}(f_{0,\infty})\leq h_{A}(g_{0,\infty})+\log nk.$
Let $\mathscr{A}$ be any finite open cover of $X$. Given $y\in Y$,
and let $\mathscr{A}_y$ denote a cover of $\cup_{j=1}^{k}\phi_j^{-1}(y)$
by at most $nk$ sets of $\mathscr{A}$. By the compactness of $X\setminus(\cup_{B\in\mathscr{A}_y}B)$,
there exists an open neighbourhood $V_{y}$ of $y$ with
\begin{equation}\label{66}
\bigcup_{j=1}^{k}\phi_j^{-1}(V_y)\subset\bigcup_{B\in\mathscr{A}_y}B.
\end{equation}
Let $\mathscr{B}=\{V_y: y\in Y\}$. Then $\mathscr{B}$ is an open cover of $Y$.
Let $\mathscr{C}_m$ be a minimal subcover of $\bigvee_{i=1}^{m}g_{0}^{-a_i}(\mathscr{B})$ for each $m\geq1$.
Denote
\[\mathscr{D}_m=\{\bigcap_{i=1}^{m}f_{0}^{-a_i}(B_i): \bigcap_{i=1}^{m}g_{0}^{-a_i}(V_{y_i})\in\mathscr{C}_m,
\;B_i\in\mathscr{A}_{y_i},\;1\leq i\leq m\}.\]
Fix $m\geq1$. For any $x\in X$, $\pi_{0}x\in V\in\mathscr{C}_m$, where $V=\cap_{i=1}^{m}g_{0}^{-a_i}(V_{y_i})$
for some $y_1,\cdots,y_m\in Y$. Thus,
\[\pi_{a_i}\circ f_{0}^{a_i}(x)=g_{0}^{a_i}\circ\pi_{0}(x)\in V_{y_i},\;1\leq i\leq m.\]
This, together with (\ref{66}), yields that
\[f_{0}^{a_i}(x)\in\pi_{a_i}^{-1}(V_{y_i})\subset\bigcup_{B\in\mathscr{A}_{y_i}}B,\;1\leq i\leq m.\]
So, there exists $B_i\in\mathscr{A}_{y_i}$ such that $f_{0}^{a_i}(x)\in B_i$ for each $1\leq i\leq m$,
and thus $x\in\cap_{i=1}^{m}f_{0}^{-a_i}(B_i)$.
Hence, $\mathscr{D}_m$ is an open cover of $X$ for each $m\geq1$,
and so is a subcover of $\bigvee_{i=1}^{m}f_{0}^{-a_i}\mathscr{A}$.
It is easy to verify that
\[\mathcal{N}(\mathscr{D}_m)\leq(nk)^{m}\mathcal{N}(\mathscr{C}_m),\;m\geq1.\]
Therefore,
\begin{equation*}
\begin{split}
h_{A}(f_{0,\infty})\leq h_{A}(f_{0,\infty},\mathscr{A})&=\limsup_{m\to\infty}\frac{1}{m}\log\mathcal{N}\left(\bigvee_{i=1}^{m}f_{0}^{-a_i}\mathscr{A}\right)
\leq\limsup_{m\to\infty}\frac{1}{m}\log\mathcal{N}(\mathscr{D}_m)\\
&\leq\limsup_{m\to\infty}\frac{1}{m}\log(nk)^{m}\mathcal{N}(\mathscr{C}_m)
=\log nk+\limsup_{m\to\infty}\frac{1}{m}\log\mathcal{N}(\mathscr{C}_m)\\
&\leq h_{A}(g_{0,\infty})+\log nk.
\end{split}
\end{equation*}
\end{proof}

\begin{remark}
Note that the method used in the proof of Theorem 3.1 is inspired by Proposition 2.5 in \cite{Goodman74},
which is the special case that $f_{0,\infty}=f$, $g_{0,\infty}=g$ and $k=1$ of Theorem 3.1.

\end{remark}

Motivated by Theorem 3.1 in \cite{Shao16}, we obtain an estimation of lower bound of t.s.e. for coupled-expanding systems in the next result.

\begin{definition}
$(X,f_{0,\infty})$ is called coupled-expanding in nonempty subsets $V_1,\cdots,V_N\subset X$ {\rm(}$N\geq2${\rm)}
with pairwise disjoint interiors if
\[f_n(V_i)\supset\bigcup_{j=1}^{N}V_j,\;1\leq i\leq N,\;n\geq0.\]
Furthermore, it is called strictly coupled-expanding in $V_i$, $1\leq i\leq N$,
if $\bar{V}_i\cap\bar{V}_j=\emptyset$ for all $1\leq i\neq j\leq N$, where $\bar{V}_i$ denotes
the closure of the set $V_i$ with respect to $X$.
\end{definition}

\begin{theorem}\label{8241}
Let $V_1,\cdots,V_N$ be nonempty, closed and mutually disjoint subsets of $X$, $N\geq2$.
Assume that $(X,f_{0,\infty})$ is coupled-expanding in $V_i$, $1\leq i\leq N$.
Then $h_A(f_{0,\infty})\geq\log N$ for any $A\in\mathcal{S}$.
\end{theorem}

\begin{proof}
Let $A=\{a_i\}_{i=1}^{\infty}$.
Since $V_1,\cdots,V_N$ are mutually disjoint, there exist pairwise disjoint open subsets $G_1,\cdots,G_N$
such that $V_i\subset G_i$, $1\leq i\leq N$. Set $G_{N+1}:=X\setminus(\cup_{i=1}^{N}V_i)$.
Then $\alpha=\{G_1,\cdots,G_{N+1}\}$ is an open cover of $X$.
Fix $n\geq1$. Then, for any $\omega=(i_1,\cdots,i_n)\in\{1,\cdots,N\}^{n}$, denote
\begin{equation}\label{614}
V_{\omega}=\bigcap_{k=1}^{n}f_{0}^{-a_k}(V_{i_k}),\;G_{\omega}=\bigcap_{k=1}^{n}f_{0}^{-a_k}(G_{i_k}).
\end{equation}
By the coupled-expansion of $(X,f_{0,\infty})$ in $V_i$, $1\leq i\leq N$, $V_{\omega}\neq\emptyset$.
Fix one point $x_{\omega}\in V_{\omega}$. Denote
\[S=\{x_{\omega}: \omega\in\{1,\cdots,N\}^{n}\}.\]
For any $\omega=(i_1,\cdots,i_n),\omega'=(i'_1,\cdots,i'_n)\in\{1,\cdots,N\}^{n}$ with $\omega\neq\omega'$,
there exists $1\leq l\leq N$ such that $i_l\neq i'_l$. We shall show the following three assertions hold.

(i) $x_{\omega}\neq x_{\omega'}$. If not, then $x_{\omega}=x_{\omega'}$, which implies that
$f_{0}^{a_l}(x_{\omega})\in V_{i_l}\cap V_{i'_l}$ by (\ref{614}). This contradicts to the fact that
$V_{i_l}\cap V_{i'_l}=\emptyset$.

(ii) $G_{\omega}$ is the unique open set of $\bigvee_{k=1}^{n}f_{0}^{-a_k}(\alpha)$
which contains $x_{\omega}$. Clearly, $x_{\omega}\in G_{\omega}$ by (\ref{614}).
Now, it is to show the uniqueness. Otherwise, there exists another subset
$\bigcap_{k=1}^{n}f_{0}^{-a_k}(G_{j_k})$ contains $x_{\omega}$.
Then there exists $1\leq m\leq n$ such that $i_m\neq j_m$.

{\bf Case 1.} Suppose that $1\leq j_m\leq N$. Then $f_{0}^{a_m}(x_{\omega})\in G_{i_m}\cap G_{j_m}$ by (\ref{614}),
which is a contradiction since $G_{i_m}$ and $G_{j_m}$ are disjoint.

{\bf Case 2.} Suppose that $j_m=N+1$. Then $f_{0}^{a_m}(x_{\omega})\in V_{i_m}\cap G_{N+1}$ by (\ref{614}),
which is a contradiction since $G_{N+1}\cap(\cup_{i=1}^{N}V_i)=\emptyset$.

(iii) $G_{\omega}$ and $G_{\omega'}$ are disjoint. Otherwise, there exists $z\in G_{\omega}\cap G_{\omega'}$.
Then $f_{0}^{a_l}(z)\in G_{i_l}\cap G_{i'_l}$ by (\ref{614}), which
is a contradiction since $G_{i_l}\cap G_{i'_l}=\emptyset$.

Hence, by the above three assertions one obtains that the number of open sets in any subcover of
$\bigvee_{k=1}^{n}f_{0}^{-a_k}(\alpha)$ is larger than or equal to the number of elements in $S$.
Therefore,
\[h_{A}(f_{0,\infty},\alpha)=\limsup_{n\to\infty}\frac{1}{n}\log \mathcal{N}\left(\bigvee_{k=1}^{n}f_{0}^{-a_k}(\alpha)\right)
\geq\limsup_{n\to\infty}\frac{1}{n}\log N^n=\log N.\]
Consequently, $h_{A}(f_{0,\infty})\geq\log N$.
\end{proof}

\subsection{T.S.E. bounds measure sequence entropy}

Now, we shall show that the t.s.e. of $f_{0,\infty}$ on $X$ with finite covering dimension
is no less than its corresponding measure sequence entropy.

We first recall the definition of measure sequence entropy \cite{Zhu12}.
Suppose that there exists a $f_{0,\infty}$-invariant measure $\mu\in\mathcal{M}(X)$,
ie. $\mu(B)=\mu(f_n^{-1}(B))$ for any $B\in\mathcal{B}(X)$ and $n\geq0$.
Then the quadruple $(X,\mathcal{B}(X),\mu,f_{0,\infty})$ is said to be a measure-theoretical dynamical system.
Let $A=\{a_i\}_{i=1}^{\infty}\in\mathcal{S}$. If $\xi$ is a finite partition of $X$, then the measure-theoretical sequence entropy
of $(X,\mathcal{B}(X),\mu,f_{0,\infty})$ with respect to $\xi$ and $A$ is defined by
\[h_{A,\mu}(f_{0,\infty},\xi)=\limsup_{n\to\infty}\frac{1}{n}H_{\mu}(\bigvee_{i=1}^{n}f_{0}^{-{a_i}}\xi),\]
where
\[H_{\mu}(\bigvee_{i=1}^{n}f_{0}^{-{a_i}}\xi)=-\sum_{C\in\bigvee_{i=1}^{n}f_{0}^{-{a_i}}\xi}\mu(C)\ln\mu(C).\]
The measure-theoretical sequence entropy of $(X,\mathcal{B}(X),\mu,f_{0,\infty})$ along $A$ is then defined by
\[h_{A,\mu}(f_{0,\infty})=\sup_{\xi}h_{A,\mu}(f_{0,\infty},\xi),\]
where the supremum is taken over all finite partitions of $X$.

By $(X^k,f_{0,\infty}^{(k)})$ denote the $k$-th product system of $(X,f_{0,\infty})$, where
\[X^k=\underbrace{X\times\cdots\times X}_{k},\; f_{0,\infty}^{(k)}=\{f_{n}^{(k)}\}_{n=0}^{\infty},\;
f_{n}^{(k)}=\underbrace{f_n\times\cdots\times f_n}_{k},\; k\geq1.\]
The metric
\[d_k((x_1,\cdots,x_k),(y_1,\cdots,y_k))=\max_{1\leq i\leq k}d(x_i,y_i),\;(x_1,\cdots,x_k),(y_1,\cdots,y_k)\in X^k,\]
is compatible with the product topology on $X^k$. To proceed, we need the following two lemmas.

\begin{lemma}\label{21}
Let $A\in\mathcal{S}$. Then
\[h_A(f_{0,\infty}^{(k)})=kh_A(f_{0,\infty}),\;k\geq1.\]
\end{lemma}

\begin{proof}
Let $m\geq1$, $\epsilon>0$ and $A=\{a_i\}_{i=1}^{\infty}$.
Suppose that $F\subset X$ is a set which $(m,\epsilon,A)$-spans $X$ under $f_{0,\infty}$ with $|F|=r_m(\epsilon,A,f_{0,\infty})$.
Then one can easily verify that $F^k$ is an $(m,\epsilon,A)$-spanning set of $(X^k,f_{0,\infty}^{(k)})$.
Thus,
\[r_m(\epsilon,A,f_{0,\infty}^{(k)})\leq (r_m(\epsilon,A,f_{0,\infty}))^{k}.\]
So
\[h_A(f_{0,\infty}^{(k)})=\lim_{\epsilon\to0}\limsup_{m\to\infty}\frac{1}{m}\log r_{m}(\epsilon,A,f_{0,\infty}^{(k)})
\leq k\lim_{\epsilon\to0}\limsup_{m\to\infty}\frac{1}{m}\log r_{m}(\epsilon,A,f_{0,\infty})=kh_A(f_{0,\infty}).\]

On the other hand, let $E\subset X$ be an $(m,\epsilon,A)$-separated set of $(X,f_{0,\infty})$ with $|E|=s_m(\epsilon,A,f_{0,\infty})$.
Then $E^k\subset X^k$ is an $(m,\epsilon,A)$-separated set of $(X^k,f_{0,\infty}^{(k)})$.
In fact, for any $x=(x_1,\cdots,x_k),y=(y_1,\cdots,y_k)\in E^k$ with $x\neq y$, there exists
$1\leq i_0\leq k$ such that $x_{i_0}\neq y_{i_0}$. Then there exists $1\leq j\leq m$ such that
$d(f_{0}^{a_j}(x_{i_0}),f_{0}^{a_j}(y_{i_0}))>\epsilon$,
which implies that
\[d_k\left((f_{0}^{a_j}(x_1),\cdots,f_{0}^{a_j}(x_k)),(f_{0}^{a_j}(y_1),\cdots,f_{0}^{a_j}(y_k))\right)>\epsilon.\]
Thus,
\[s_m(\epsilon,A,f_{0,\infty}^{(k)})\geq (s_m(\epsilon,A,f_{0,\infty}))^{k},\]
which yields that
\begin{equation*}
\begin{split}
h_A(f_{0,\infty}^{(k)})&=\lim_{\epsilon\to0}\limsup_{m\to\infty}\frac{1}{m}\log r_{m}(\epsilon,A,f_{0,\infty}^{(k)})\\
&\geq k\lim_{\epsilon\to0}\limsup_{m\to\infty}\frac{1}{m}\log s_{m}(\epsilon,A,f_{0,\infty})=kh_A(f_{0,\infty}).
\end{split}
\end{equation*}
\end{proof}

\begin{lemma}\label{8101}
Let $\mu\in\mathcal{M}(X)$ be $f_{0,\infty}$-invariant. Then for any $A\in\mathcal{S}$,
\[h_{A,\mu}(f_{0,\infty}^{(k)})=kh_{A,\mu}(f_{0,\infty}),\;k\geq1.\]
\end{lemma}

\begin{proof}
We only prove the case $k=2$, others can be shown inductively.
Let $\xi_1$ and $\xi_2$ be two finite partitions of $X$. Then
\begin{equation}\label{8102}
\begin{split}
h_{A,\mu}(f_{0,\infty}^{(2)},\xi_1\times\xi_2)
&=\limsup_{n\to\infty}\frac{-1}{n}\sum_{m,j}\mu(C_m)\mu(D_j)\ln(\mu(C_m)\mu(D_j))\\
&=\limsup_{n\to\infty}\frac{-1}{n}\left(\sum_{m,j}\mu(D_j)\mu(C_m)\ln\mu(C_m)+\sum_{m,j}\mu(C_m)\mu(D_j)\ln\mu(D_j)\right)\\
&=\limsup_{n\to\infty}\frac{-1}{n}\left(\sum_{m}\mu(C_m)\ln\mu(C_m)+\sum_{j}\mu(D_j)\ln\mu(D_j)\right)\\
&\leq\limsup_{n\to\infty}\frac{-1}{n}\sum_{m}\mu(C_m)\ln\mu(C_m)+\limsup_{n\to\infty}\frac{-1}{n}\sum_{j}\mu(D_j)\ln\mu(D_j)\\
&=h_{A,\mu}(f_{0,\infty},\xi_1)+h_{A,\mu}(f_{0,\infty},\xi_2)\leq 2h_{A,\mu}(f_{0,\infty}),
\end{split}
\end{equation}
where $\{C_m\}$ are the members of $\bigvee_{i=1}^{n}f_{0}^{-{a_i}}\xi_1$, and $\{D_j\}$ are the members of $\bigvee_{i=1}^{n}f_{0}^{-{a_i}}\xi_2$. This implies that
\[h_{A,\mu}(f_{0,\infty}^{(2)})\leq 2h_{A,\mu}(f_{0,\infty}).\]
On the other hand, consider the case $\xi_1=\xi_2=\xi$ in (\ref{8102}), one obtains that
\[h_{A,\mu}(f_{0,\infty}^{(2)},\xi\times\xi)=2h_{A,\mu}(f_{0,\infty},\xi)\]
for any finite partition $\xi$ of $X$. Thus,
\[h_{A,\mu}(f_{0,\infty}^{(2)})\geq2h_{A,\mu}(f_{0,\infty}).\]
Hence, $h_{A,\mu}(f_{0,\infty}^{(2)})=2h_{A,\mu}(f_{0,\infty})$.
\end{proof}

Let $\alpha$ be an open cover of $X$.
Recall that the order of $\alpha$ is the maximum number of sets in $\alpha$ with nonempty intersection and denoted by ord$\alpha$.
$X$ has covering dimension at most $n$ if any open cover of $X$ has an open refinement of order at most $n+1$ and denoted by dim$X$.

\begin{theorem}\label{7281}
Suppose that $X$ has finite covering dimension and $\mu\in\mathcal{M}(X)$ is $f_{0,\infty}$-invariant. Then for any $A\in\mathcal{S}$, $h_{A}(f_{0,\infty})\geq h_{A,\mu}(f_{0,\infty})$.
\end{theorem}

\begin{proof}
Let dim$X=k$ and $A=\{a_i\}_{i=1}^{\infty}$. Take a finite open cover $\alpha=\{B_i\}_{i=1}^{t}$ of $X$
with ord$\alpha\leq k+1$. Then construct a partition $\xi=\{A_i\}_{i=1}^{t}$ of $X$ satisfying that
$\bar{A_i}\subset B_i$, $1\leq i\leq t$. For any $x\in X$, choose an open neighborhood $C_x$ which intersects at most
$k+1$ members of $\xi$. Let $\beta$ be a finite subcover of $\{C_x: x\in X\}$ with Lebesgue number $2\epsilon>0$.
Let $n\geq1$ and $F$ be an $(n,\epsilon,A)$-spanning set of $(X,f_{0,\infty})$ with $|F|=r_n(\epsilon,A,f_{0,\infty})$.
Denote
\[S_n=\{(i_1,\cdots,i_n)\in\{1,\cdots,t\}^n: {\rm there}\; {\rm exists}\; x\in X\; {\rm such}\; {\rm that}\; f_{0}^{a_l}(x)\in A_{i_l},\; 1\leq l\leq n\},\]
and
\[T_n=\{(i_1,\cdots,i_n)\in\{1,\cdots,t\}^n: \forall\;1\leq l\leq n,\;\exists\; y\in F\; {\rm and}\; B_{l,y}\in\beta,\;{\rm s.t.} \;f_{0}^{a_l}(x)\in A_{i_l}\cap B_{l,y}\}.\]
Then the number of sets in the partition $\bigvee_{i=1}^{n}f_{0}^{-{a_i}}\xi$ is less than or equal to $|S_n|$.
It is to show that $S_n\subset T_n$. For any $(i_1,\cdots,i_n)\in S_n$, there exists $x\in X$ such that
$f_{0}^{a_l}(x)\in A_{i_l}$ for any $1\leq l\leq n$. Then there exists $y\in F$ such that
$d(f_{0}^{a_l}(x),f_{0}^{a_l}(y))\leq\epsilon$ for any $1\leq l\leq n$.
Thus, for any $1\leq l\leq n$, there exists $B_{l,y}\in \beta$ such that
\[f_{0}^{a_l}(x)\in A_{i_l}\cap B(f_{0}^{a_l}(y),\epsilon)\subset A_{i_l}\cap B_{l,y}.\]
So, $(i_1,\cdots,i_n)\in T_n$. This, together with the choice of $\beta$, yields that
\begin{equation*}
|S_n|\leq|T_n|\leq(k+1)^n|F|=(k+1)^nr_n(\epsilon,A,f_{0,\infty}).
\end{equation*}
Then we have
\begin{equation*}
H_{\mu}(\bigvee_{i=1}^{n}f_{0}^{-{a_i}}\xi)\leq\ln|S_n|\leq n\ln(k+1)+\ln r_n(\epsilon,A,f_{0,\infty}),
\end{equation*}
which yields that
\begin{equation}\label{631}
h_{A,\mu}(f_{0,\infty},\xi)\leq\ln(k+1)+h_{A}(f_{0,\infty}).
\end{equation}

Take $\delta>0$. Then there exists a finite Borel partition $\eta=\{D_i\}_{i=1}^{s}$ of $X$ with
\begin{equation}\label{734}
h_{A,\mu}(f_{0,\infty},\eta)>h_{A,\mu}(f_{0,\infty})-\frac{1}{2}\delta,
\end{equation}
and one can also choose $\delta_1>0$ such that for any Borel partition $\eta'=\{D'_i\}_{i=1}^{s}$
of $X$,
\begin{equation}\label{733}
\mu(D_i\bigtriangleup D'_i)<\delta_1, 1\leq i\leq s,\;\Rightarrow\; H_{\mu}(\eta/\eta')<\frac{1}{2}\delta.
\end{equation}
Since $\mu$ is regular, there exists a closed set $K_i\subset D_i$ satisfying that
\[\mu(D_i\setminus K_i)<s^{-1}\delta_1,\;1\leq i\leq s.\]
Denote $\alpha'=\{O_i\}_{i=1}^{s}$, where $O_i=X\setminus(\cup_{j\neq i}K_j)$, $1\leq i\leq s$.
Then $\alpha'$ is an open cover of $X$. Let $\alpha$ be a finite refinement of $\alpha'$ with ord$\alpha\leq k+1$.
The above shows that there exists a Borel partition $\xi$ of $X$ which is a refinement of $\alpha$ and
satisfies (\ref{631}). Denote $\gamma'=\{E_i\}_{i=1}^{s}$, where
\[E_i=\bigcup\{G\in\xi: G\cap K_i\neq\emptyset\},\;1\leq i\leq s-1,\;E_s=X\setminus(\cup_{j=1}^{s-1}E_j).\]
Then $\gamma'$ is a Borel partition of $X$ and $\gamma'$ is a refinement of $\xi$.
Since $K_i\subset E_i$, $1\leq i\leq s$, we have
\begin{equation}\label{731}
\mu(D_i\setminus E_i)\leq\mu(D_i\setminus K_i)<s^{-1}\delta_1,\;1\leq i\leq s.
\end{equation}
We also have
\[E_i\setminus D_i\subset E_i\bigcap (\bigcup_{j\neq i}D_j)=\bigcup_{j\neq i}(E_i\cap D_j)\subset\bigcup_{j\neq i}(D_j\setminus E_j),\]
This, together with (\ref{731}), implies that
\begin{equation}\label{732}
\mu(E_i\setminus D_i)\leq(s-1)s^{-1}\delta_1=(1-s^{-1})\delta_1,\;1\leq i\leq s.
\end{equation}
It follows from (\ref{731}) and (\ref{732}) that
$\mu(D_i\bigtriangleup E_i)<\delta_1$. Using (\ref{733}), we have that
\begin{equation}\label{735}
H_{\mu}(\eta/\gamma')<\frac{\delta}{2}.
\end{equation}
Now,
\begin{equation*}
\begin{split}
H_{\mu}(\bigvee_{i=1}^{n}f_{0}^{-a_i}\eta)&\leq H_{\mu}(\bigvee_{i=1}^{n}f_{0}^{-a_i}(\eta\vee\gamma'))
=H_{\mu}(\bigvee_{i=1}^{n}f_{0}^{-a_i}\gamma')+H_{\mu}(\bigvee_{i=1}^{n}f_{0}^{-a_i}(\eta/\gamma'))\\
&\leq H_{\mu}(\bigvee_{i=1}^{n}f_{0}^{-a_i}\gamma')+nH_{\mu}(\eta/\gamma').
\end{split}
\end{equation*}
Thus,
\begin{equation}\label{736}
h_{A,\mu}(f_{0,\infty},\eta)\leq h_{A,\mu}(f_{0,\infty},\gamma')+H_{\mu}(\eta/\gamma').
\end{equation}
By (\ref{631}), (\ref{734}), (\ref{735}) and (\ref{736}), we have that
\begin{equation}
\begin{split}
h_{A,\mu}(f_{0,\infty})&<h_{A,\mu}(f_{0,\infty},\eta)+\frac{\delta}{2}
\leq h_{A,\mu}(f_{0,\infty},\gamma')+H_{\mu}(\eta/\gamma')+\frac{\delta}{2}\\
&\leq h_{A,\mu}(f_{0,\infty},\gamma')+\delta
\leq h_{A,\mu}(f_{0,\infty},\xi)+\delta
\leq \ln(k+1)+h_{A}(f_{0,\infty})+\delta.
\end{split}
\end{equation}
Since $\delta$ is arbitrary,
\[h_{A,\mu}(f_{0,\infty})\leq h_{A}(f_{0,\infty})+\ln(k+1).\]
For any $n\geq1$, let $Y=X^n$ and $g_{0,\infty}=\{f_i^{(n)}\}_{i=0}^{\infty}$, where $f_i^{(n)}=\underbrace{f_i\times\cdots\times f_i}_{n}$.
Since dim$Y\leq nk$, the above argument, together with Lemmas \ref{21} and \ref{8101}, implies that
\[nh_{A,\mu}(f_{0,\infty})=h_{A,\mu}(g_{0,\infty})\leq h_{A}(g_{0,\infty})+\ln(nk+1)=nh_{A}(f_{0,\infty})+\ln(nk+1).\]
Let $n\to\infty$. Then we have that
\[h_{A,\mu}(f_{0,\infty})\leq h_{A}(f_{0,\infty}).\]
\end{proof}

\begin{remark}
It is worth mentioning that Theorem \ref{7281} is inspired by Theorem 3.1 in \cite{Goodman74}, and we extend it to NDSs.
\end{remark}

\section{The supremum t.s.e.}

For a homeomorphism $f: X\to X$, Huang and Ye \cite{Huang09} introduced the notion of
maximal pattern entropy $h^{*}(f)$, one of its equivalent definitions is that
\[h^{*}(f)=\sup_{A\in\mathcal{S}}h_A(f),\]
and thus it is called the ``supremum t.s.e." of $(X,f)$ in \cite{Snoha20}.
Here, we generalize this concept to NDSs and define
\[h^{*}(f_{0,\infty})=\sup_{A\in\mathcal{S}}h_A(f_{0,\infty}),\]
and we shall study some properties of it in this section.

By $(X,f_{0,\infty}^{n})$ denote the $n$-th compositions system of $(X,f_{0,\infty})$, where
\[f_{0,\infty}^{n}=\{f_{kn}^{n}\}_{k=0}^{\infty},\; \;f_{kn}^{n}=f_{kn+n-1}\circ\cdots\circ f_{kn},\;k\geq0,\;n\geq1.\]
We need the following lemma.

\begin{lemma}\label{6151}
Let $A=\{a_i\}_{i=1}^{\infty}\in\mathcal{S}$, $nA=\{na_i\}_{i=1}^{\infty}$ and
$\sigma: {\mathds{Z}_{+}}^{\mathds{N}}\to{\mathds{Z}_{+}}^{\mathds{N}}$ be the shift map defined by $\sigma(a_1,a_2,\cdots)=(a_2,a_3,\cdots)$.  Then
\begin{itemize}
\item[{\rm(i)}]
\begin{equation}\label{7242}
h_{A}(f_{0,\infty}^{n})=h_{nA}(f_{0,\infty}),\;n\geq1;
\end{equation}
\item[{\rm(ii)}]
\begin{equation}\label{7241}
h_A(f_{0,\infty})=h_{\sigma^k(A)}(f_{0,\infty}),\;k\geq1.
\end{equation}
\end{itemize}
\end{lemma}

\begin{proof}
(i) Let $\alpha$ be an open cover of $X$ and $n\geq1$. Then
\begin{equation*}
\begin{split}
h_A(f_{0,\infty}^{n},\alpha)
&=\limsup_{k\to\infty}\frac{1}{k}\log\mathcal{N}\left((f_{(a_{1}-1)n}^{n}\circ\cdots \circ f_{0}^{n})^{-1}\alpha\vee\cdots\vee(f_{(a_{k}-1)n}^{n}\circ\cdots \circ f_{0}^{n})^{-1}\alpha\right)\\
&=\limsup_{k\to\infty}\frac{1}{k}\log\mathcal{N}\left(f_{0}^{-na_1}\alpha\vee\cdots\vee f_{0}^{-na_k}\alpha\right)
=h_{nA}(f_{0,\infty},\alpha).
\end{split}
\end{equation*}
Hence, $h_{A}(f_{0,\infty}^{n})=h_{nA}(f_{0,\infty})$.

(ii) Let $\alpha$ be an open cover of $X$. It is easy to see that
\begin{equation*}
\begin{split}
h_{A}(f_{0,\infty},\alpha)
\leq\limsup_{n\to\infty}\frac{1}{n}\log\mathcal{N}(\bigvee_{i=1}^{k}f_{0}^{-a_i}\alpha)
+\limsup_{n\to\infty}\frac{1}{n}\log\mathcal{N}(\bigvee_{i=k+1}^{n}f_{0}^{-a_i}\alpha)
=h_{\sigma^k(A)}(f_{0,\infty},\alpha),
\end{split}
\end{equation*}
which implies that $h_A(f_{0,\infty})\leq h_{\sigma^k(A)}(f_{0,\infty})$. On the other hand, note that
\begin{equation*}
h_{A}(f_{0,\infty},\alpha)=\limsup_{n\to\infty}\frac{1}{n}\log\mathcal{N}(\bigvee_{i=1}^{n}f_{0}^{-a_i}\alpha)
\geq\limsup_{n\to\infty}\frac{1}{n}\log\mathcal{N}(\bigvee_{i=2}^{n}f_{0}^{-a_i}\alpha)=h_{\sigma(A)}(f_{0,\infty},\alpha).
\end{equation*}
One can inductively get that $h_A(f_{0,\infty})\geq h_{\sigma^k(A)}(f_{0,\infty})$.
\end{proof}

It is well known that $h(f^n)=nh(f)$ for each $n\geq1$, where $f: X\to X$ is a continuous map.
This equality was extended to NDSs in \cite{Kolyada96} provided that $f_{0,\infty}$ is equi-continuous on $X$:
\begin{equation}\label{7201}
h(f_{0,\infty}^{n})=nh(f_{0,\infty}), \;n\geq1.
\end{equation}
However, $h_{A}(f^n)\neq nh_{A}(f)$ in general \cite{Lemancz85}.
In the case that $A$ is some special sequence, we have the following result.

\begin{proposition}\label{6152}
Let $A=\{k^i\}_{i=1}^{\infty}$ for some $k\geq1$. Then
\[h_{A}(f_{0,\infty}^{k^n})=h_{A}(f_{0,\infty}),\;n\geq1.\]
\end{proposition}

\begin{proof}
Fix $n\geq1$. By (\ref{7242}) and (\ref{7241}), one gets that
\[h_{A}(f_{0,\infty}^{k^n})=h_{k^nA}(f_{0,\infty})=h_{\sigma^n(A)}(f_{0,\infty})=h_{A}(f_{0,\infty}).\]
\end{proof}

\begin{proposition}
Let $f_{0,\infty}$ be equi-continuous on $X$. Then $h(f_{0,\infty})>0$ implies that $h^{*}(f_{0,\infty})=+\infty$.
\end{proposition}

\begin{proof}
Fix $n\geq1$. By (\ref{7242}) and (\ref{7201}), we have
\[h_{n\mathds{Z}^{+}}(f_{0,\infty})=h(f_{0,\infty}^{n})=nh(f_{0,\infty}).\]
This implies that $h^{*}(f_{0,\infty})=+\infty$ since $n$ is arbitrary.
\end{proof}


Motivated by Proposition \ref{6152}, we obtain an equality for supremum t.s.e.,
which extends Theorem 2.3 (2) in \cite{Huang09} (see also Theorem 1 in \cite{Wu17}) to NDSs.

\begin{theorem}\label{sequence entropy equality}
Let $f_{0,\infty}$ be equi-continuous on $X$. Then
\[h^{*}(f_{0,\infty}^{n})=h^{*}(f_{0,\infty}),\;n\geq1.\]
\end{theorem}

\begin{proof}
It follows from (\ref{7242}) that $h^{*}(f_{0,\infty}^{n})\leq h^{*}(f_{0,\infty})$.
On the other hand, we take an increasing sequence of positive integers
\begin{equation}\label{6141}
B=\{j\in\mathds{Z}_{+}: [jn,jn+(n-1)]\cap A\neq\emptyset\}=\{b_i\}_{i=1}^{\infty}.
\end{equation}
Without loss of generality, we assume that $b_1\geq1$ by (\ref{7241}).
Since $f_{0,\infty}$ is equi-continuous, for any $m\geq1$, there exists $0<\delta_m<\frac{1}{m}$ such that for any $x,y\in X$,
\begin{equation}\label{6142}
d(x,y)<\delta_m\Rightarrow d(f_{p}^{i}(x),f_{p}^{i}(y))<\frac{1}{m},\;0\leq i\leq n-1,\;\forall\;p\geq0.
\end{equation}
Let $k\geq1$ and $E\subset X$ be an $(k,\frac{1}{m},A)$-separated set under $f_{0,\infty}$ with $|E|=s_k(\frac{1}{m},A,f_{0,\infty})$.
Then for any $x\neq y\in E$, there exists $1\leq i\leq k$ such that $d(f_{0}^{a_i}(x),f_{0}^{a_i}(y))>\frac{1}{m}$.
It follows from (\ref{6141}) and (\ref{6142}) that there exists $1\leq r\leq k$ (in fact, $r\leq i$) such that
$a_i\in[b_{r}n,b_{r}n+(n-1)]$ and
\[d(f_{(b_{r}-1)n}^{n}\circ\cdots\circ f_{0}^{n}(x),f_{(b_{r}-1)n}^{n}\circ\cdots\circ f_{0}^{n}(y))
=d(f_{0}^{b_{r}n}(x),f_{0}^{b_{r}n}(y))>\delta_m,\]
which implies that $E$ is also an $(k,\delta_m,B)$-separated set under $f_{0,\infty}^n$.
Hence,
\begin{equation*}
h_{A}(f_{0,\infty})=\lim_{m\to\infty}\limsup_{k\to\infty}\frac{1}{k}\log s_k(\frac{1}{m},A,f_{0,\infty})
\leq\lim_{m\to\infty}\limsup_{k\to\infty}\frac{1}{k}\log s_k(\delta_m,B,f_{0,\infty}^n)=h_{B}(f_{0,\infty}^n),
\end{equation*}
and consequently, $h^{*}(f_{0,\infty})\leq h^{*}(f_{0,\infty}^{n})$.
\end{proof}

\begin{proposition}\label{sequence entropy inequality}
\[h^{*}(f_{i,\infty})\leq h^{*}(f_{j,\infty}),\;0\leq i\leq j<\infty.\]
\end{proposition}

\begin{proof}
Let $\mathscr{A}$ be an open cover of $X$ and $A=\{a_i\}_{i=1}^{\infty}\in\mathcal{S}$.
Without loss of generality, we assume that $a_1\geq2$ by (\ref{7241}). Then
\begin{equation*}
\begin{split}
\mathcal{N}(f_{i}^{-a_1}\mathscr{A}\vee f_{i}^{-a_2}\mathscr{A}\vee\cdots\vee f_{i}^{-a_n}\mathscr{A})
&=\mathcal{N}\left(f_{i}^{-1}(f_{i+1}^{-(a_{1}-1)}\mathscr{A}\vee f_{i+1}^{-(a_{2}-1)}\mathscr{A}\vee\cdots\vee f_{i+1}^{-(a_{n}-1)}\mathscr{A})\right)\\
&\leq\mathcal{N}(f_{i+1}^{-(a_{1}-1)}\mathscr{A}\vee f_{i+1}^{-(a_{2}-1)}\mathscr{A}\vee\cdots\vee f_{i+1}^{-(a_{n}-1)}\mathscr{A}).
\end{split}
\end{equation*}
Thus
\begin{equation*}
\begin{split}
h_A(f_{i,\infty},\mathscr{A})
&=\limsup_{n\to\infty}\frac{1}{n}\log\mathcal{N}\left(\bigvee_{k=1}^{n}f_{i}^{-a_k}\mathscr{A}\right)
\leq\limsup_{n\to\infty}\frac{1}{n}\log\mathcal{N}\left(\bigvee_{k=1}^{n}f_{i+1}^{-(a_{k}-1)}\mathscr{A}\right)\\
&=h_{A'}(f_{i,\infty},\mathscr{A}),
\end{split}
\end{equation*}
where $A'=\{a_{i}-1\}_{i=1}^{\infty}$. This yields that
$h^{*}(f_{i,\infty})\leq h^{*}(f_{j,\infty})$.
\end{proof}

\begin{proposition}\label{8246}
Let $f_{0,\infty}$ be a sequence of equi-continuous and surjective maps.
Then for any $A\in\mathcal{S}$,
\[h_{A}(f_{i,\infty})\geq h_{A}(f_{j,\infty}),\;0\leq i\leq j<\infty,\]
and thus
\[h^{*}(f_{i,\infty})=h^{*}(f_{j,\infty}),\;0\leq i\leq j<\infty.\]
\end{proposition}

\begin{proof}
Fix $k\geq0$. It is easy to observe that $(X,f_{k,\infty})$ is topologically $\{f_n\}_{n=k}^{\infty}$-equi-semiconjugate to
$(X,f_{k+1,\infty})$, since $\{f_n\}_{n=k}^{\infty}$ is a sequence of equi-continuous and surjective maps such
that the following diagram commutes:
\begin{center}
 \begin{tikzpicture}[scale=0.58]
 \draw [->] (0, 0)--(3, 0)node[right]{};
       \node (a) at (3.4,0) {\tiny$X$};
       \node (a) at (-0.4,0) {\tiny$X$};
       \node (a) at (1.5,0.4) {\tiny$f_k$};
    \draw    [->] (3.8, 0)--(6.8, 0)node[right]{};
        \node (a) at (7.2,0) {\tiny$X$};
       \node (a) at (5.4,0.4) {\tiny$f_{k+1}$};
        \draw    [->] (7.6, 0)--(10.6, 0)node[right]{};
       \draw    [->] (-0.4, -0.4)--(-0.4, -3.4)node[right]{};
       \node (a) at (-0.4,-3.8) {\tiny$X$};
       \draw [->] (0, -3.8)--(3, -3.8)node[right]{};
       \node (a) at (1.5,-3.4) {\tiny$f_{k+1}$};
       \node (a) at (3.4,-3.8) {\tiny$X$};
       \draw    [->] (3.4, -0.4)--(3.4, -3.4)node[right]{};
       \draw    [->] (3.8, -3.8)--(6.8, -3.8)node[right]{};
       \node (a) at (7.2,-3.8) {\tiny$X$};
       \node (a) at (5.4,-3.4) {\tiny$f_{k+2}$};
       \draw    [->] (7.6, -3.8)--(10.6, -3.8)node[right]{};
        \draw    [->] (7.2, -0.4)--(7.2, -3.4)node[right]{};
         \node (a) at (0,-1.8) {\tiny$f_{k}$};
         \node (a) at (4,-1.8) {\tiny$f_{k+1}$};
         \node (a) at (7.8,-1.8) {\tiny$f_{k+2}$};
          \node (a) at (11.5,0) {\tiny$\cdots$};
          \node (a) at (11.5,-3.8) {\tiny$\cdots$};
 \end{tikzpicture}
\end{center}
that is,
\[f_{k+n+1}\circ f_{k+n}=f_{k+n+1}\circ f_{k+n},\;n\geq0.\]
By Lemma \ref{22}, we have
\[h_A(f_{k,\infty})\geq h_A(f_{k+1,\infty}),\;A\in\mathcal{S}.\]
Thus,
\[h_{A}(f_{i,\infty})\geq h_{A}(f_{j,\infty}),\;0\leq i\leq j<\infty.\]
Hence,
\[h^{*}(f_{i,\infty})\geq h^{*}(f_{j,\infty}),\;0\leq i\leq j<\infty.\]
This, together with Proposition \ref{sequence entropy inequality}, implies that
\[h^{*}(f_{i,\infty})=h^{*}(f_{j,\infty}),\;0\leq i\leq j<\infty.\]
\end{proof}

\begin{lemma}\cite{Bali99a,Wu17}\label{commu}
Let $f,g: X\to X$ be continuous maps and $A\in\mathcal{S}$. Then

{\rm(i)} $h^{*}(f\circ g)=h^{*}(g\circ f)$;

{\rm(ii)} $h_A(f\circ g)=h_A(g\circ f)$ if $f,g$ are surjective.
\end{lemma}

We have two consequences of Lemma \ref{commu}.

\begin{corollary}
Let $A\in\mathcal{S}$. Then for each $n\geq1$ and $1\leq i\leq n$,
\begin{itemize}
\item[{\rm(i)}] \[h^{*}(f_{n}\circ\cdots\circ f_{0})=h^{*}(f_{i-1}\circ\cdots\circ f_{0}\circ f_{n}\circ\cdots\circ f_{i});\]

\item[{\rm(ii)}] \[h_{A}(f_{n}\circ\cdots\circ f_{0})=h_{A}(f_{i-1}\circ\cdots\circ f_{0}\circ f_{n}\circ\cdots\circ f_{i}),\]
if $f_i$ is surjective for each $0\leq i\leq n$.
\end{itemize}
\end{corollary}

\begin{proof}
By Lemma \ref{commu} (i), one gets that
\begin{equation*}
\begin{split}
&h^{*}(f_{n}\circ\cdots\circ f_{0})=h^{*}\left((f_{n}\circ\cdots\circ f_{i})\circ(f_{i-1}\circ\cdots\circ f_{0})\right)\\
&=h^{*}((f_{i-1}\circ\cdots\circ f_{0})\circ(f_{n}\circ\cdots\circ f_{i}))
=h^{*}(f_{i-1}\circ\cdots\circ f_{0}\circ f_{n}\circ\cdots\circ f_{i}).
\end{split}
\end{equation*}
The second equality can be shown similarly by Lemma \ref{commu} (ii).
\end{proof}

\section{T.S.E. on the induced space of probability measures}

In this section, we shall study the relations between zero t.s.e. (resp. positive t.s.e.)
of $(X,f_{0,\infty})$ and that of $(\mathcal{M}(X),\hat{f}_{0,\infty})$. In particular, we obtain some interesting
comparisons between t.s.e. of NDSs and that of classical dynamical systems as applications.

\subsection{Relations on zero t.s.e.}

First, we need the following two lemmas.

\begin{lemma}\cite{Shao22}\label{7271}
$h_A(f_{0,\infty})\leq h_A(\hat{f}_{0,\infty})$ for any $A\in\mathcal{S}$.
\end{lemma}

Let $l_1^n=\mathbb{R}^n$ endowed with the norm $\left\|x\right\|=\sum_{i=1}^{n}|x_i|$
and $l_{\infty}^n=\mathbb{R}^n$ endowed with the norm $\left\|x\right\|_{\infty}=\max_{1\leq i\leq n}\{|x_i|\}$
for each $n\geq1$, and let $\mathcal{C}(X,\mathds{R})$ be the space of bounded real-valued continuous functions
on $X$ with the norm $\left\|\varphi\right\|=\sup_{x\in X}|\varphi(x)|$ for any $\varphi\in\mathcal{C}(X,\mathds{R})$.

\begin{lemma}\cite{Glasner}\label{321}
For given constants $\epsilon>0$ and $b>0$, there exist a positive integer $n_0$ and a constant $c>0$ such that
for all $n>n_0$, if $\phi$ is a linear mapping from $l_1^m$ to $l_{\infty}^n$ of norm
\[\left\|\phi\right\|=\sup\{\left\|\phi(x)\right\|_{\infty}: x\in l_1^m,\;\left\|x\right\|\leq1\}\leq1,\]
and if $\phi(B_1(l_1^m))$ contains more than $2^{bn}$ points that are $\epsilon$-separated, then $m\geq2^{cn}$,
where $B_1(l_1^m)=\{y\in l_1^m: \left\|y\right\|\leq1\}$.
\end{lemma}

Similarly to a single map, we call $(X,f_{0,\infty})$ to be null if $h_A(f_{0,\infty})=0$ for any $A\in\mathcal{S}$,
i.e. $h^{*}(f_{0,\infty})=0$.

\begin{theorem}\label{sequence entropy}
Let $A\in\mathcal{S}$. Then $h_A(f_{0,\infty})=0$ if and only if $h_A(\hat{f}_{0,\infty})=0$.
Consequently, $(X,f_{0,\infty})$ is null if and only if $(\mathcal{M}(X),\hat f_{0,\infty})$ is null.
\end{theorem}

\begin{proof}
By Lemma \ref{7271} we have that $h_A(\hat{f}_{0,\infty})=0$ implies that $h_A(f_{0,\infty})=0$.
It is to show that $h_A(f_{0,\infty})=0$ implies that $h_A(\hat{f}_{0,\infty})=0$.
Otherwise, $h_A(\hat{f}_{0,\infty})>0$.
Then there exists $\epsilon_0>0$ such that for any $0<\epsilon<\epsilon_0$,
\[\limsup_{n\to\infty}\frac{1}{n}\log s_{n}(\epsilon,A,\hat{f}_{0,\infty})>0.\]
Thus, there exist $b>0$ and an increasing sequence of positive integers $\{n_i\}_{i=1}^{\infty}$ such that
\[s_{n_i}(\epsilon,A,\hat{f}_{0,\infty})\geq e^{bn_i},\;i\geq1.\]
Let $\{g_n\}_{n=1}^{\infty}$ be a sequence in $\mathcal{C}(X,\mathds{R})$ satisfying that
$\left\|g_n\right\|\leq1$ for each $n\geq1$ and
\begin{equation*}
\mathcal{P}_{d}(\mu,\nu)=\sum_{n=1}^{\infty}\frac{|\int g_nd\mu-\int g_nd\nu|}{2^n},\;\;\mu,\nu\in\mathcal{M}(X).
\end{equation*}
For any $0<\epsilon<\epsilon_0$, there exists a positive integer $K$ such that
\begin{equation}\label{7311}
\sum_{n=K+1}^{+\infty}\frac{1}{2^{n-1}}<\frac{\epsilon}{2}.
\end{equation}
By the continuity of $\{g_n\}_{n=1}^{\infty}$, one can find $\delta>0$ such that
\begin{equation}\label{33}
d(x,y)<\delta\Rightarrow|g_n(x)-g_n(y)|<\frac{\epsilon}{9},\;\;1\leq n\leq K.
\end{equation}
Take an open cover $\alpha$ of $X$ with $d(\alpha)<\delta$.
We shall show that $h_{A}(f_{0,\infty},\alpha)>0$.

First, we shall construct a linear mapping $\phi: l_1^{k_{n_i}}\to l_{\infty}^{Kn_i}$ with $\left\|\phi\right\|\leq1$,
where $A=\{a_t\}_{t=1}^{\infty}$ and $k_{n_i}=\mathcal{N}\left(\bigvee_{t=1}^{n_i}f_{0}^{-a_t}\alpha\right)$ for each $i\geq1$.
Take a subcover $\beta=\{B_1,B_2,\cdots,B_{k_{n_i}}\}$ of $\bigvee_{t=1}^{n_i}f_{0}^{-a_t}\alpha$
with the minimal cardinality. Denote
\[C_1=B_1,\;C_2=B_2\setminus B_1,\;\cdots,C_{k_{n_i}}=B_{k_{n_i}}\setminus \left(\bigcup_{j=1}^{k_{n_i}-1}B_j\right).\]
Then $\bigcup_{j=1}^{k_{n_i}}C_j=X$ and each $C_j\neq\emptyset$.
Take $y_j\in C_j$ for each $1\leq j\leq k_{n_i}$.
Define $\phi: l_1^{k_{n_i}}\to l_{\infty}^{Kn_i}$ by
\[\phi(x_1,\cdots,x_{k_{n_i}})
=\left(\frac{1}{2^n}\sum_{j=1}^{k_{n_i}}g_n(f_{0}^{a_t}(y_j))x_j\right)_{1\leq n\leq K,\;1\leq t\leq n_i}.\]
It is easy to verify that $\phi$ is a linear mapping and $\left\|\phi\right\|\leq1$.

Second, we shall show that $\phi(B_1(l_1^{k_{n_i}}))$ contains more than $e^{bn_i}$ points that are $\epsilon/(9\cdot2^{K})$-separated.
Let $E$ be an $(n_i,\epsilon,A)$-separated set of $(\mathcal{M}(X),\hat{f}_{0,\infty})$ with $|E|=s_{n_i}(\epsilon,A,\hat{f}_{0,\infty})$.
Thus, for any $\mu\neq\nu\in E$, there exists $1\leq t_0\leq n_i$ such that
\[\mathcal{P}_{d}(\hat{f}_{0}^{a_{t_0}}(\mu),\hat{f}_{0}^{a_{t_0}}(\nu))
=\sum_{n=1}^{\infty}\frac{|\int g_nd(\hat{f}_{0}^{a_{t_0}}(\mu))-\int g_nd(\hat{f}_{0}^{a_{t_0}}(\nu))|}{2^n}>\epsilon,\]
which, together with (\ref{7311}), implies that
\begin{equation}\label{7301}
\sum_{n=1}^{K}\frac{|\int g_nd(\hat{f}_{0}^{a_{t_0}}(\mu))-\int g_nd(\hat{f}_{0}^{a_{t_0}}(\nu))|}{2^n}>\frac{\epsilon}{2}.
\end{equation}
Define a map $\psi: E\to B_1(l_1^{k_{n_i}})$ by
\[\psi(\mu)=(\mu(C_1),\mu(C_2),\cdots,\mu(C_{k_{n_i}})).\]
It is clear that $\phi(\psi(E))\subset\phi(B_1(l_1^{k_{n_i}}))$.
We claim that for any $\mu\neq\nu\in E$, $\phi(\psi(\mu))$ and $\phi(\psi(\nu))$ are $\epsilon/(9\cdot2^{K})$-separated.
Otherwise, for any $1\leq n\leq K$ and $1\leq t\leq n_i$,
\begin{equation*}
\frac{1}{2^n}\left|\sum_{j=1}^{k_{n_i}}g_n(f_{0}^{a_t}(y_j))\mu(C_j)
-\sum_{j=1}^{k_{n_i}}g_n(f_{0}^{a_t}(y_j))\nu(C_j)\right|
\leq\frac{\epsilon}{9\cdot2^K}.
\end{equation*}
This yields that
\begin{equation}\label{7312}
\left|\sum_{j=1}^{k_{n_i}}g_n(f_{0}^{a_t}(y_j))\mu(C_j)
-\sum_{j=1}^{k_{n_i}}g_n(f_{0}^{a_t}(y_j))\nu(C_j)\right|
\leq\frac{\epsilon}{9}.
\end{equation}
On the other hand, we have
\begin{align*}
&\left|\int g_n(f_{0}^{a_{t}}(x))d\mu(x)-\int g_n(f_{0}^{a_{t}}(x))d\nu(x)\right|
\leq\left|\int g_n(f_{0}^{a_{t}}(x))d\mu(x)-\sum_{j=1}^{k_{n_i}}g_n(f_{0}^{a_t}(y_j))\mu(C_j)\right|\\
&+\left|\sum_{j=1}^{k_{n_i}}g_n(f_{0}^{a_t}(y_j))\mu(C_j)
-\sum_{j=1}^{k_{n_i}}g_n(f_{0}^{a_t}(y_j))\nu(C_j)\right|
+\left|\sum_{j=1}^{k_{n_i}}g_n(f_{0}^{a_t}(y_j))\nu(C_j)
-\int g_n(f_{0}^{a_{t}}(x))d\nu(x)\right|.
\end{align*}
Since $\beta=\{B_1,B_2,\cdots,B_{k_{n_i}}\}$ is a subcover of $\bigvee_{t=1}^{n_i}f_{0}^{-a_t}\alpha$,
there exist $A_1^{j},\cdots,A_{n_i}^{j}\in\alpha$ such that $B_j=\bigcap_{t=1}^{n_i}f_{0}^{-a_t}(A_t^{j})$
for any $1\leq j\leq k_{n_i}$. It follows from the fact that $C_j\subset B_j$ that
\[x,y\in C_j\Rightarrow f_{0}^{a_t}(x),f_{0}^{a_t}(y)\in A_t^{j},\;1\leq t\leq n_i,\;1\leq j\leq k_{n_i}.\]
This, together with the fact that $d(\alpha)<\delta$ and (\ref{33}), implies that for any $1\leq j\leq k_{n_i}$,
\begin{equation*}
x,y\in C_j\Rightarrow|g_n(f_{0}^{a_t}(x))-g_n(f_{0}^{a_t}(y))|<\frac{\epsilon}{9},\;\;1\leq n\leq K,\;1\leq t\leq n_i.
\end{equation*}
Thus,
\begin{align*}
\left|\int g_n(f_{0}^{a_{t}}(x))d\mu(x)-\sum_{j=1}^{k_{n_i}}g_n(f_{0}^{a_t}(y_j))\mu(C_j)\right|
&=\left|\sum_{j=1}^{k_{n_i}}\int_{C_j} g_n(f_{0}^{a_{t}}(x))d\mu(x)
-\sum_{j=1}^{k_{n_i}}\int_{C_j} g_n(f_{0}^{a_{t}}(y_j))d\mu(x)\right|\\
&\leq\sum_{j=1}^{k_{n_i}}\int_{C_j}\left|g_n(f_{0}^{a_t}(x))-g_n(f_{0}^{a_t}(y_j))\right|d\mu(x)
\leq\frac{\epsilon}{9}.
\end{align*}
Similarly, we obtain that
\[\left|\int g_n(f_{0}^{a_{t}}(x))d\nu(x)-\sum_{j=1}^{k_{n_i}}g_n(f_{0}^{a_t}(y_j))\nu(C_j)\right|\leq\frac{\epsilon}{9}.\]
This, together with (\ref{7312}), implies that
\begin{align*}
\left|\int g_n(f_{0}^{a_{t}}(x))d\mu(x)-\int g_n(f_{0}^{a_{t}}(x))d\nu(x)\right|\leq\frac{\epsilon}{3},\;1\leq n\leq K,\;1\leq t\leq n_i.
\end{align*}
This contradicts to (\ref{7301}).
Note that $E$ contains more than $e^{bn_i}$ points and $\phi\circ\psi$ is injective on $E$.
Hence, $\phi(\psi(E))$ contains more than $e^{bn_i}$ points that are $\epsilon/(9\cdot2^{K})$-separated.

Finally, we show that $h_A(f_{0,\infty})>0$.
We have proved above that $\phi$ is a linear mapping from $l_1^{k_{n_i}}$ to $l_{\infty}^{Kn_i}$ with
$\left\|\phi\right\|\leq1$, and $\phi(B_1(l_1^{k_{n_i}}))$ contains more than $e^{bn_i}$ points that are
$\epsilon/(9\cdot2^{K})$-separated, then $m\geq2^{cn}$.
Thus, by Lemma \ref{321}, we have that there exists a constant $c>0$ such that
\[k_{n_i}\geq2^{cn_i},\; \;i\geq1.\]
Hence,
\[h_A(f_{0,\infty})\geq\limsup_{i\to\infty}\frac{1}{n_i}\log\mathcal{N}\left(\bigvee_{t=1}^{n_i}f_{0}^{-a_t}\alpha\right)
=\limsup_{i\to\infty}\frac{\log k_{n_i}}{n_i}\geq\limsup_{i\to\infty}\frac{\log2^{cn_i}}{n_i}=c\log2>0.\]
This is a contradiction.
Therefore, $h_A(f_{0,\infty})=0$ implies that $h_A(\hat{f}_{0,\infty})=0$.
\end{proof}

\subsection{Relations on positive t.s.e.}

\begin{lemma}\cite{Liu20}\label{31a}
Let $k\geq1$. The map $\pi: X^k\to\mathcal{M}(X)$ defined by
\begin{equation}\label{811}
\pi((x_1,\cdots,x_k))=\frac{1}{\sum_{i=1}^{k}2^i}\sum_{i=1}^{k}2^i\delta_{x_i}
\end{equation}
is injective.
\end{lemma}

\begin{theorem}\label{sequence entropy 1}
Let $A\in\mathcal{S}$. Then $h_A(f_{0,\infty})>0$ if and only if $h_A(\hat{f}_{0,\infty})=+\infty$.
Consequently, $h^{*}(f_{0,\infty})>0$ if and only if $h^{*}(\hat{f}_{0,\infty})=+\infty$.
\end{theorem}

\begin{proof}
It is easy to see that $h_A(\hat{f}_{0,\infty})=+\infty$ implies that $h_A(f_{0,\infty})>0$
by Theorem \ref{sequence entropy}. On the other hand, suppose that $h_A(f_{0,\infty})>0$.
Given $k\geq1$. It is easy to verify that
\begin{equation}\label{32}
\pi\circ f_{n}^{(k)}=\hat{f}_n\circ\pi,\;\;n\geq0,
\end{equation}
where $\pi$ is defined in (\ref{811}), and $\pi$ is continuous and injective by Lemma \ref{31a}.
This, together with (\ref{32}), implies that $(X^k,f_{0,\infty}^{(k)})$ is topologically $\pi$-equi-conjugate to
$(\Omega,\hat{f}_{0,\infty})$, where
\[\Omega=\{\frac{1}{\sum_{i=1}^{k}2^i}\sum_{i=1}^{k}2^i\delta_{x_i}: (x_1,\cdots,x_k)\in X^k\}.\]
Thus, by Lemmas \ref{22} and \ref{21}, we have that
\[h_A(\hat{f}_{0,\infty})\geq h_A(\hat{f}_{0,\infty},\Omega)=h_A(f_{0,\infty}^{(k)})=kh_A(f_{0,\infty}).\]
Hence, $h_A(f_{0,\infty})>0$ implies that $h_A(\hat{f}_{0,\infty})=+\infty$.

Suppose that $h^{*}(f_{0,\infty})>0$. Then there exists $A\in\mathcal{S}$ such that $h_A(f_{0,\infty})>0$,
which yields that $h_A(\hat{f}_{0,\infty})=+\infty$, and thus $h^{*}(\hat{f}_{0,\infty})=+\infty$.
On the other hand, suppose that $h^{*}(\hat{f}_{0,\infty})=+\infty$. It is to show that $h^{*}(f_{0,\infty})>0$.
Otherwise, $h^{*}(f_{0,\infty})=0$, then $h^{*}(\hat{f}_{0,\infty})=0$ by Theorem \ref{sequence entropy},
which is a contradiction. Thus $h^{*}(f_{0,\infty})>0$.
\end{proof}

\subsection{Some applications}

The next several results shows some interesting comparisons between entropies of NDSs and that of classical
dynamical systems.

\begin{theorem}\label{822}
{\rm(i)} Let $f$ be a continuous surjective self-map on the compact metric space $X$.
If $h(f)>0$, then $h_{A}(f)>0$, and thus $h_{A}(\hat f)=+\infty$ for any $A\in\mathcal{S}$.

{\rm(ii)} There exists a NDS $(X, f_{0,\infty})$ with $h(f_{0,\infty})>0$,
but $h_{A}(f_{0,\infty})=h_{A}(\hat f_{0,\infty})=0$ for some $A\in\mathcal{S}$,
where $X=[0,1]$, $id$ is the identity map on $X$, and
\[f_{0,\infty}=\{g, \frac{1}{2}id, g^2, \frac{1}{4}id, \cdots,  g^n, \frac{1}{2^n}id,\cdots\}.\]
\end{theorem}

\begin{proof}
{\rm(i)} By Theorem 5.4 in \cite{Huang05}, $h(f)>0$ implies that $h_{A}(f)>0$ for any $A\in\mathcal{S}$,
and thus $h_{A}(\hat f)=+\infty$ by Theorem \ref{sequence entropy 1}.

{\rm(ii)} It is easy to verify that $f_{0}^{2n-1}=g^n$ for each $n\geq1$. Then
\[s_{2n}(\epsilon,f_{0,\infty})\geq s_{n+1}(\epsilon,g),\;n\geq1.\]
Thus,
\[h(f_{0,\infty})\geq\frac{1}{2}h(g)=\frac{1}{2}\log2>0.\]
Let $A=\{2n\}_{n=1}^{\infty}$. Since
\[f_{0}^{2n}(x)\leq\frac{1}{2^n},\;n\geq1,\;x\in X,\]
we have that
\[h_{A}(f_{0,\infty})=\limsup_{n\to\infty}\frac{1}{n}\log s_{n}(\epsilon,A,f_{0,\infty})=0.\]
Hence, $h_{A}(\hat f_{0,\infty})=0$ by Theorem \ref{sequence entropy}.
\end{proof}

\begin{lemma}\cite{Kolyada96}\label{7273}
Let $f_{0,\infty}$ be a sequence of continuous self-maps on a compact metric space $X$ converging uniformly to $f$. Then
$h(f_{0,\infty})\leq h(f)$.
\end{lemma}

Recall that $(X,f_{0,\infty})$ is topologically transitive if $\{n\geq 1:f_{0}^{n}(U)\cap V\neq\emptyset\}\neq\emptyset$
for any two nonempty open subsets $U,V\subset X$.

\begin{theorem}\label{8242}
\begin{itemize}
\item[{\rm(i)}] If $(I,f)$ is topologically transitive, then $h_{A}(\hat f)=+\infty$ for any $A\in\mathcal{S}$.

\item[{\rm(ii)}] There exists a topologically transitive NDS $(I,f_{0,\infty})$ satisfying that
$h_{A}(\hat f_{0,\infty})=0$ for some $A\in\mathcal{S}$.
\end{itemize}
\end{theorem}

\begin{proof}
(i) By Corollary 3.6 in \cite{BC87} we have that $h(f)\geq {1\over 2}\log 2>0$.
This, together with Theorem 5.4 in \cite{Huang05}, implies that $h_{A}(f)>0$, and thus
$h_{A}(\hat f)=+\infty$ for any $A\in\mathcal{S}$ by Theorem \ref{sequence entropy 1}.

(ii) Since $h_{A}(f_{0,\infty})$ is the standard topological entropy of $(X,f_{0,\infty})$ when $A=\mathds{Z}_{+}$,
(ii) is derived from Theorem 5.4 in \cite{Shao22}, we give its detailed construction for completeness.

Let $I=[0,1]$. First, we construct a family of functions  $F_m:I\to I$ for $m>0$. Divide $I$ into $m$ intervals
\[J_i\triangleq[a_i,a_{i+1}],\;0\leq i\leq m-1,\]
where $a_i={i\over m}$. For any $0\leq i\leq m-1$, put $c_i,d_i\in J_i$ with
\[c_i=a_i+{1\over3m},\; d_i=a_i+{2\over3m},\; d_{-1}=0, \;c_m=1.\]
The map $F_m$ is the connect-the-dots map
such that
\[F_m(a_i)=a_i,\; F_m(c_i)=c_{i+1},\;F_m(d_i)=d_{i-1},\; 0\leq i\leq m-1.\]
Then $(I,F_m)$ is exact for any fixed $m\geq 1$.
Next, inductively define the maps $\{f_n\}_{n=0}^{\infty}$. Let \[\mathcal{A}_{n}\triangleq\{[{{i}\over{2^{n}}},{{i+1}\over{2^{n}}}]:i=0,\cdots,2^{n}-1\}.\]
Then there exists $s_1\geq 1$ such that $\underbrace{F_{1}\circ\cdots\circ F_{1}}_{s_1}(J)=I$ for  $J\in\mathcal{A}_{1}$. Denote \[f_i\triangleq F_1,\; 0\leq i\leq s_1-1.\]
Assume that we have already defined $s_1<s_2<\cdots<s_n$ and $\{f_j\}_{j=0}^{s_n-1}$ such that $f_{0}^{s_k}(J)=I$ for  $J\in\mathcal{A}_{k}$ and $1\leq k\leq n$.
Let us inductively define $s_{n+1}$ and  $\{f_j\}_{j=s_n}^{s_{n+1}-1}$.
For any $J\in\mathcal{A}_{n+1}$, there exists
$l\geq 1$ such that $\underbrace{F_{n+1}\circ\cdots\circ F_{n+1}}_{l}(f_{0}^{s_n}(J))=I$. Denote
\[s_{n+1}\triangleq s_n+l,\;f_j\triangleq F_{n+1},\;s_n\leq j\leq s_{n+1}-1.\]
By the above construction, we get that

(1) $f_n$ converges uniformly to ${id}$ on $I$, which means that $h(f_{0,\infty})\leq h(id)=0$ by Lemma \ref{7273},
and thus $h(\hat f_{0,\infty})=0$ by Theorem \ref{sequence entropy};

(2) $f_{0}^{s_{n}}(J)=I$ for each $n\geq 1$ and $J\in\mathcal{A}_{n}$, which yields that
$(I, f_{0,\infty})$ is topologically transitive.
\end{proof}

\begin{proposition}\label{8243}
{\rm(i)} Let $f$ be a continuous self-map on a compact metric space $(X,d)$.
Assume that for any $x\in X$, there exists a fixed point $y$ of $f$ such that $\lim_{n\to\infty}f^{n}(x)=y$.
Then $h(f)=h(\hat f)=0$.

{\rm(ii)} There exists a NDS $([0,1], f_{0,\infty})$
such that $0, 1$ are fixed points and all others are asymptotic to $0$, but $h(f_{0,\infty})\geq\log 2$
and $h(\hat f_{0,\infty})=+\infty$.
\end{proposition}

\begin{proof}
{\rm(i)} By Proposition 2 in \cite{Lamparta}, we have that $h(f)=0$, and thus $h(\hat f)=0$ by Theorem \ref{sequence entropy}.

{\rm(ii)} It follows from Theorem 4 in \cite{Bali12} that $h(f_{0,\infty})\geq\log 2$.
Thus, $h(\hat f_{0,\infty})=+\infty$ by Theorem \ref{sequence entropy 1}.
\end{proof}

\begin{remark}
Theorem 4 in \cite{Bali12} mainly shows that positive topological entropy does not necessarily imply Li-Yorke chaos
for NDSs, which is also of big difference from that of classical dynamical systems.
\end{remark}

\section{Multi-sensitivity implies positive or infinite t.s.e.}

In this section, we shall study whether multi-sensitivity of $(X,f_{0,\infty})$ imply positive
or infinite t.s.e. Recall that $(X,f_{0,\infty})$ is called multi-sensitive
if there exists $\delta>0$ such that for any $k\geq 1$ and any nonempty open subsets
$V_1,\cdots,V_k\subset X$, there exists $n\geq 1$ satisfying that $d(f_{0}^{n}(V_i))>\delta$
for each $1\leq i\leq k$.

\begin{theorem}\label{8181}
Assume that $f_0$ is surjective. If $(X,f_{0,\infty})$ is multi-sensitive,
then there exists $A\in\mathcal{S}$ such that $h^{*}(f_{0,\infty})\geq h_{A}(f_{0,\infty})\geq\log 2$,
and thus $h_A(\hat{f}_{0,\infty})=h^{*}(\hat{f}_{0,\infty})=+\infty$.
\end{theorem}

\begin{proof}
Let $\epsilon>0$ be the multi-sensitivity constant of $(X,f_{0,\infty})$.
Take $z_1\neq z_2\in X$. Denote
\begin{equation}\label{411}
0<\xi<\min\{\epsilon,d(z_1,z_2)/3\},\; U_i=B(z_i,\xi),\; V_i=f_{0}^{-1}(U_i),\;i=1,2.
\end{equation}
In the following, we shall inductively define a sequence $A=\{a_n\}_{n=1}^{\infty}\in\mathcal{S}$ satisfying that
$h_{A}(f_{0,\infty})\geq\log 2$.

Let $a_1=1$. Fix $x_1\in V_1$ and $x_2\in V_2$. It follows from (\ref{411}) that
\begin{equation}\label{421}
d(f_{0}^{a_1}(x_1),f_{0}^{a_1}(x_2))\geq d(z_1,z_2)-d(f_{0}(x_1),z_1)-d(f_{0}(x_2),z_2)>\xi.
\end{equation}
Then $\{x_1,x_2\}$ is $(1,\xi,A)$-separated, and thus
\[s_{1}(\xi,A,f_{0,\infty})\geq 2.\]
By the continuity of $f_0$, there exists $\delta_1>0$ such that for any $x,y\in X$,
\[d(x,y)<\delta_1\Rightarrow d(f_0(x),f_0(y))<\epsilon.\]
Choose $0<\zeta_1<\delta_1/2$ satisfying that $B(x_i,\zeta_1)\subset V_i$, $i=1,2$.
Since $(X,f_{0,\infty})$ is multi-sensitive, there exists $a_2>a_1$ such that $d(f_{0}^{a_2}(B(x_i,\zeta_1)))>\epsilon$ for $i=1,2$.
So, there exist $x_{(i,1)},x_{(i,2)}\in B(x_i,\zeta_1)$ such that
\begin{equation}\label{431}
d(f_{0}^{a_2}(x_{(i,1)}),f_{0}^{a_2}(x_{(i,2)}))>\epsilon>\xi,\;i=1,2.
\end{equation}
This, together with (\ref{421}), implies that $\{x_{(1,1)},x_{(1,2)},x_{(2,1)},x_{(2,2)}\}$ is $(2,\xi,A)$-separated, and thus
\[s_{2}(\xi,A,f_{0,\infty})\geq 2^{2}.\]
By the continuity of $f_{0}^{t}$, $2\leq t\leq a_2$, there exists $\delta_2>0$ such that for any $x,y\in X$,
\[d(x,y)<\delta_2\Rightarrow d(f_{0}^{t}(x),f_{0}^{t}(y))<\epsilon,\;2\leq t\leq a_2.\]
Choose $0<\zeta_2<\delta_2/2$ satisfying that $B(x_{(i,j)},\zeta_2)\subset B(x_i,\zeta_1)$, $i,j=1,2$.
Since $(X,f_{0,\infty})$ is multi-sensitive, there exists $a_3>a_2$ such that
$d(f_{0}^{a_3}(B(x_{(i,j)},\zeta_2)))>\epsilon$ for any $1\leq i,j\leq 2$. Then there exist
$x_{(i,j,1)},x_{(i,j,2)}\in B(x_{(i,j)},\zeta_2)$ such that
\begin{equation}\label{441}
d(f_{0}^{a_3}(x_{(i,j,1)}),f_{0}^{a_3}(x_{(i,j,2)}))>\epsilon>\xi,\;1\leq i,j\leq 2.
\end{equation}
It follows from (\ref{421})-(\ref{441}) that the set of points $\{x_{(i,j,k)}: 1\leq i,j,k\leq 2\}$ is $(3,\xi,A)$-separated, and thus
\[s_{3}(\xi,A,f_{0,\infty})\geq 2^3.\]
By repeating the above process, one obtains that there exist $a_n>a_{n-1}>\cdots>a_1$ such that the set of points
$\{x_{(i_1,\cdots,i_n)}: 1\leq i_1,\cdots,i_n\leq 2\}$ is $(n,\xi,A)$-separated, and thus
\[s_{n}(\xi,A,f_{0,\infty})\geq 2^n.\]
Hence,
\[h^{*}(f_{0,\infty})\geq h_{A}(f_{0,\infty})=\lim_{\xi\to0}\limsup_{n\to\infty}\frac{1}{n}\log s_{n}(\xi,A,f_{0,\infty})\geq\log 2.\]
It follows from Theorem \ref{sequence entropy} that $h_A(\hat{f}_{0,\infty})=+\infty$, and consequently,
$h^{*}(\hat{f}_{0,\infty})=+\infty$.
\end{proof}

By making a slight modification to the proof of Theorem \ref{8181}, one obtains the following result,
we give its detailed proof for completeness.

\begin{theorem}\label{8182}
Let $X$ be a locally connected and compact metric space and $f_0$ be surjective.
If $(X,f_{0,\infty})$ is multi-sensitive, then there exists $A\in\mathcal{S}$ such that $h_{A}(f_{0,\infty})=+\infty$,
and thus $h_{A}(\hat f_{0,\infty})=h^{*}(\hat f_{0,\infty})=h^{*}(f_{0,\infty})=+\infty$.
\end{theorem}

\begin{proof}
Let $2\epsilon>0$ be the multi-sensitivity constant of $(X,f_{0,\infty})$.
Fix $n\geq2$ and take $n$ distinct points $z_1,\cdots,z_n\in X$.
Let
\begin{equation}\label{41}
0<\xi<\min\{\epsilon/(n-1),\min_{1\leq i\neq j\leq n} d(z_i,z_j)/3\}.
\end{equation}
Now, we shall inductively define a sequence $A=\{a_k\}_{k=1}^{\infty}\in\mathcal{S}$ satisfying that
\begin{equation*}
s_{m}(\xi,A,f_{0,\infty})\geq n^m,\;m\geq1.
\end{equation*}
If this equality holds, then
\[h_{A}(f_{0,\infty})=\lim_{\xi\to0}\limsup_{m\to\infty}\frac{1}{m}\log s_{m}(\xi,A,f_{0,\infty})\geq\log n.\]
Since $n$ is arbitrary, $h_{A}(f_{0,\infty})=+\infty$, and thus $h^{*}(f_{0,\infty})=+\infty$.
By Theorem \ref{sequence entropy 1}, we also have $h_{A}(\hat f_{0,\infty})=h^{*}(\hat f_{0,\infty})=+\infty$.

Let $a_1=1$ and take $z_i'\in f_{0}^{-1}(z_i)$, $1\leq i\leq n$.
It is clear that
\[d(f_{0}(z_i'),f_{0}(z_j'))>\xi,\;1\leq i\neq j\leq n.\]
Then the set of points $\{z_i': 1\leq i\leq n\}$ is $(1,\xi,A)$-separated, and thus
\[s_{1}(\xi,A,f_{0,\infty})\geq n.\]
Denote $U_i=B(z_i,\xi)$, $1\leq i\leq n$. Then
\begin{equation}\label{42}
\inf\{d(x,y): x\in U_i,\;y\in U_j\}>\xi,\;1\leq i\neq j\leq n.
\end{equation}
By the continuity of $f_{0}$, there exists $\delta_1>0$ such that for any $x,y\in X$,
\[d(x,y)<\delta_1\Rightarrow d(f_{0}(x),f_{0}(y))<2\epsilon.\]
Since $X$ is locally connected, there exists nonempty open and connected subset
$V_i\subset f_{0}^{-1}(U_i)$ with $d(V_i)<\delta_1$, $1\leq i\leq n$.
By the multi-sensitivity of $(X,f_{0,\infty})$, there exists $a_2>a_1$ such that
$d(f_{0}^{a_2}(V_i))>2\epsilon$ for any $1\leq i\leq n$.
Then, there exist $x_{(i1)},x_{(i1)}'\in V_i$ such that $d(f_{0}^{a_2}(x_{(i1)}),$ $f_{0}^{a_2}(x_{(i1)}'))>2\epsilon$.
Define a map $f: V_i\to\mathds{R}$ by $f(x)=d(f_{0}^{a_2}(x),f_{0}^{a_2}(x_{(i1)}))$ for any $x\in V_i$.
It is clear that $f(x_{(i1)})=0$, $f(x_{(i1)}')>2\epsilon$ and $f$ is continuous on $V_i$.
Since $X$ is locally connected,$f(V_i)$ is an interval, and then $[0,2\epsilon]\subset f(V_i)$.
Thus, we can choose $n-1$ distinct points $x_{(i2)},\cdots,x_{(in)}\in V_i$ such that
\[f(x_{(ij)})=d(f_{0}^{a_2}(x_{(ij)}),f_{0}^{a_2}(x_{(i1)}))=\frac{2(j-1)\epsilon}{n-1},\;2\leq j\leq n,\]
which implies that
\begin{equation}\label{43}
\begin{split}
d(f_{0}^{a_2}(x_{(ij)}),f_{0}^{a_2}(x_{(ik)}))&\geq d(f_{0}^{a_2}(x_{(ik)}),f_{0}^{a_2}(x_{(i1)}))-d(f_{0}^{a_2}(x_{(ij)}),f_{0}^{a_2}(x_{(i1)}))\\
&=\frac{2(k-j)\epsilon}{n-1}\geq\frac{2\epsilon}{n-1},\;\;2\leq j<k\leq n,\;1\leq i\leq n.
\end{split}
\end{equation}
Moreover, by the continuity of $f_{0}^{t}$, $2\leq t\leq a_2$, there exists $\delta_2>0$ such that
$V_{ij}:=B(x_{(ij)},\delta_2/2)\subset V_i$ is connected and for any $x,y\in X$ and $2\leq t\leq a_2$,
\begin{equation}\label{44}
d(x,y)<\delta_2\Rightarrow d(f_{0}^{t}(x),f_{0}^{t}(y))<\frac{\epsilon}{2(n-1)}.
\end{equation}
We claim that for any $(i_1,i_2)\neq(j_1,j_2)\in\{1,\cdots,n\}^2$,
\begin{equation}\label{831}
\min_{1\leq k\leq 2}\{d(f_{0}^{a_k}(x),f_{0}^{a_k}(y))\}>\xi,\;x\in V_{i_1i_2},\;y\in V_{j_1,j_2}.
\end{equation}
In fact, if $i_1\neq j_1$. Then $x\in V_{i_1}\subset f_{0}^{-1}(U_{i_1})$ and $y\in V_{j_1}\subset f_{0}^{-1}(U_{j_1})$.
This, together with (\ref{42}), implies that
\[d(f_{0}(x),f_{0}(y))>\xi.\]
If $i_1=j_1$ and $i_2\neq j_2$. By (\ref{43}) and (\ref{44}) we have that
\begin{equation*}
\begin{split}
d(f_{0}^{a_2}(x),f_{0}^{a_2}(y))&\geq d(f_{0}^{a_2}(x_{i_{1}i_{2}}),f_{0}^{a_2}(x_{j_{1}j_{2}}))-d(f_{0}^{a_2}(x_{i_{1}i_{2}}),f_{0}^{a_2}(x))
-d(f_{0}^{a_2}(x_{j_{1}j_{2}}),f_{0}^{a_2}(y))\\
&>\frac{2\epsilon}{n-1}-\frac{\epsilon}{n-1}=\frac{\epsilon}{n-1}>\xi.
\end{split}
\end{equation*}
Then the set of points $\{x_{(i,j)}: 1\leq i,j\leq n\}$ is $(2,\xi,A)$-separated, and thus
\[s_{2}(\xi,A,f_{0,\infty})\geq n^2.\]

By the multi-sensitivity of $(X,f_{0,\infty})$, there exists $a_3>a_2$ such that
$d(f_{0}^{a_3}(V_{ij}))>2\epsilon$ for any $1\leq i,j\leq n$.
Then, there exist $x_{(ij1)},x_{(ij1)}'\in V_{ij}$ such that $d(f_{0}^{a_3}(x_{(ij1)}),$ $f_{0}^{a_3}(x_{(ij1)}'))>2\epsilon$.
Define a map $f: V_{ij}\to\mathds{R}$ by $f(x)=d(f_{0}^{a_3}(x),f_{0}^{a_2}(x_{(ij1)}))$ for any $x\in V_{ij}$.
It is clear that $f(x_{(ij1)})=0$, $f(x_{(ij1)}')>2\epsilon$.
Since $X$ is locally connected, $[0,2\epsilon]\subset f(V_{ij})$.
Thus, we can choose $n-1$ distinct points $x_{(ij2)},\cdots,x_{(ijn)}\in V_{ij}$ such that
\[f(x_{(ijk)})=d(f_{0}^{a_3}(x_{(ijk)}),f_{0}^{a_3}(x_{(ij1)}))=\frac{2(j-1)\epsilon}{n-1},\;2\leq k\leq n,\]
which implies that
\begin{equation}\label{832}
\begin{split}
d(f_{0}^{a_3}(x_{(ijk_1)}),f_{0}^{a_3}(x_{(ijk_2)}))&\geq d(f_{0}^{a_3}(x_{(ijk_1)}),f_{0}^{a_3}(x_{(ij1)}))-d(f_{0}^{a_3}(x_{(ijk_2)}),f_{0}^{a_2}(x_{(ij1)}))\\
&=\frac{2(k_2-k_1)\epsilon}{n-1}\geq\frac{2\epsilon}{n-1},\;\;2\leq k_1<k_2\leq n,\;1\leq i,j\leq n.
\end{split}
\end{equation}
Moreover, by the continuity of $f_{0}^{t}$, $a_2+1\leq t\leq a_3$, there exists $\delta_3>0$ such that
$V_{ijk}:=B(x_{(ijk)},\delta_3/2)\subset V_{ij}$ is connected and for any $x,y\in X$ and $a_2+1\leq t\leq a_3$,
\begin{equation}\label{44}
d(x,y)<\delta_3\Rightarrow d(f_{0}^{t}(x),f_{0}^{t}(y))<\frac{\epsilon}{2(n-1)}.
\end{equation}
We claim that for any $(i_1,i_2,i_3)\neq(j_1,j_2,j_3)\in\{1,\cdots,n\}^3$,
\begin{equation*}
\min_{1\leq k\leq 3}\{d(f_{0}^{a_k}(x),f_{0}^{a_k}(y))\}>\xi,\;x\in V_{i_1i_2i_3},\;y\in V_{j_1j_2j_3}.
\end{equation*}
In fact, If $(i_1,i_2)\neq(j_1,j_2)$. Then $x\in V_{i_1i_2}$ and $y\in V_{j_1j_2}$.
By (\ref{831}), we have that
\[\min_{1\leq k\leq 2}\{d(f_{0}^{a_k}(x),f_{0}^{a_k}(y))\}>\xi.\]
If $(i_1,i_2)=(j_1,j_2)$ and $i_3\neq j_3$. By (\ref{832}), we have that
\begin{equation*}
d(f_{0}^{a_3}(x),f_{0}^{a_3}(y))>\xi.
\end{equation*}
Thus
\[s_{3}(\xi,A,f_{0,\infty})\geq n^3.\]

By repeating the above procedure, we can find $a_m>\cdots>a_2>a_1$ and $n^{m}$ distinct points
$x_{(i_1,\cdots,i_{m})}\in V_{i_1\cdots i_{m}}\subset V_{i_1\cdots i_{m-1}}$, such that
for any $(i_1,\cdots,i_{m})\neq(j_1,\cdots,j_{m})\in\{1,\cdots,n\}^{m}$,
\begin{equation*}
\min_{1\leq k\leq m}\{d(f_{0}^{a_k}(x),f_{0}^{a_k}(y))\}>\xi,\;x\in V_{(i_1,\cdots,i_{m})},\;y\in V_(j_1,\cdots,j_{m}),
\end{equation*}
which yields that
\[s_{m}(\xi,A,f_{0,\infty})\geq n^{m}.\]
\end{proof}

\begin{remark}
Theorems \ref{8181} and \ref{8182} are mainly inspired by Proposition 6.1 in \cite{Huang16} and Theorem 2 in \cite{Wu17}.
\end{remark}

\section*{Acknowledgement}
This research was partially supported by the Natural Science Foundation of Jiangsu Province of China (No. BK20200435)
and the Fundamental Research Funds for the Central Universities (No. NS2021053).

\end{document}